\documentclass[11pt, letterpaper]{amsart}
\usepackage{amsfonts,amsmath, amsthm, amssymb, latexsym, epsfig}
\usepackage[all]{xy}
\usepackage{xspace}
\usepackage{comment}
\usepackage{setspace}
\usepackage{xcolor}
\usepackage{enumerate}
\usepackage[mathscr]{eucal}
\usepackage{todonotes}

	\advance \topmargin by -\headheight
	\advance \topmargin by -\headsep

	\evensidemargin \oddsidemargin
	\newcommand{\ftn}[3]{ #1 \colon #2 \rightarrow #3 }
		\newcommand{\setof}[2]{\ensuremath{\left\{ #1 \: : \: #2 \right\}}}

		\newcommand{\norm}[1]{ \left\| #1 \right\| }

	\newcommand{\kk}{\ensuremath{\mathit{KK}}\xspace}
	\newcommand{\coker}{\ensuremath{\operatorname{coker}}}
	
	\newcommand{\Hom}{\ensuremath{\operatorname{Hom}}}

	\newcommand{\id}{\ensuremath{\operatorname{id}}}
	
	\newcommand{\multialg}[1]{\mathcal{M}(#1)\xspace}
	\newcommand{\corona}[1]{\mathcal{Q}(#1)\xspace}
	
	\newcommand{\Z}{\ensuremath{\mathbb{Z}}\xspace}
	\newcommand{\C}{\ensuremath{\mathbb{C}}\xspace}
	
	\newcommand{\R}{\ensuremath{\mathbb{R}}\xspace}
	\newcommand{\N}{\ensuremath{\mathbb{N}}\xspace}
	
	\newcommand{\K}{\ensuremath{\mathbb{K}}\xspace}
	\newcommand{\catc}{\mathfrak{C}^{*}\text{-}\mathfrak{alg}}
	
	\newcommand{\Mod}{\mathfrak{Mod}}	
	\newcommand{\Ext}{\ensuremath{\operatorname{Ext}}}

	\theoremstyle{plain}
	\newtheorem{thm}{Theorem}[section]
	\newtheorem{lemma}[thm]{Lemma}
	\newtheorem{theor}[thm]{Theorem}
	\newtheorem{propo}[thm]{Proposition}
	\newtheorem{corol}[thm]{Corollary}

	\theoremstyle{definition}
	\newtheorem{defin}[thm]{Definition}

	\numberwithin{equation}{section}
	\numberwithin{figure}{section}

\begin{document}
	\title[$C^{*}$-algebras over the one-point compactification of $\N$]{Classification of tight $C^{*}$-algebras over the one-point compactification of $\N$}
	\author{James Gabe}
        \address{Department of Mathematical Sciences \\
        University of Copenhagen\\
        Universitetsparken~5 \\
        DK-2100 Copenhagen, Denmark}
        \email{gabe@math.ku.dk}
        
	\author{Efren Ruiz}
        \address{Department of Mathematics\\University of Hawaii,
Hilo\\200 W. Kawili St.\\
Hilo, Hawaii\\
96720-4091 USA}
        \email{ruize@hawaii.edu}
        \date{\today}
	


	\begin{abstract}
	We prove a strong classification result for a certain class of $C^{*}$-algebras with primitive ideal space $\widetilde{\N}$, where $\widetilde{\N}$ is the one-point compactification of $\N$.  This class contains the class of graph $C^{*}$-algebras with primitive ideal space $\widetilde{\N}$.  Along the way, we prove a universal coefficient theorem with ideal-related $K$-theory for $C^{*}$-algebras over $\widetilde{\N}$ whose $\infty$ fiber has torsion-free $K$-theory.
	\end{abstract}

        \maketitle

\section{Introduction}

Continuous fields of $C^{*}$-algebras appear naturally in the theory of $C^{*}$-algebras since every $C^{*}$-algebra with a Hausdorff primitive ideal space is isomorphic to a continuous field of $C^{*}$-algebras with simple fibers (see \cite{Fell:Acta61} and \cite{BlanchardKirchberg:JOT04}).  The problem of classifying these $C^{*}$-algebras is an important and classical problem in the theory.  In general, these algebras are very far from being locally trivial.  In a classical paper \cite{DixmierDouady:BSMF63}, Jacques Dixmier and Adrien Douady classified a certain class of continuous fields of $C^{*}$-algebras over $X$ (continuous trace $C^{*}$-algebras) by associating to each such $C^{*}$-algebra an element in the third cohomology group $\check{H}^{3} ( X , \Z )$.  By \cite[Theorem~5.2]{Dadarlat:Adv09}, if $X$ is zero-dimensional, then the section algebra of the continuous field of $C^{*}$-algebras studied by Dixmier and Douady are AF-algebras.  Thus, by Elliott's classification of AF-algebras \cite{af}, they are classified by their $K_{0}$-groups.  

Since $K$-theory has proven to be a very successful invariant for classifying $C^{*}$-algebras, it is natural to ask ``to what extent does $K$-theory classify continuous fields of $C^{*}$-algebras with simple fibers that are classifiable via $K$-theory''.  There has been recent progress in this direction, for example, Marius Dadarlat and Cornel Pasnicu in \cite{DadarlatPasnicu:JFA05} classified continuous fields of $C^{*}$-algebras over locally compact, metrizable, zero-dimensional spaces for which the fibers are purely infinite simple $C^{*}$-algebras.  Partial results have also been obtained involving continuous fields of $C^{*}$-algebras over non-zero dimensional spaces (see \cite{DadarlatElliott:CMP07}, \cite{DadarlatElliottNiu:CJM11}, \cite{BentmanDadarlat:arXiv:1306.1691}, and \cite{Bentman:arXiv:1308.2126}).  In all of the above results, either all the fibers are purely infinite or all the fibers are AF-algebras. 

In this paper, we consider the classification of $C^{*}$-algebras whose primitive ideal space is Hausdorff and fibers are of mixed type.  In fact, we consider the classification problem for $C^{*}$-algebras whose primitive ideal space is $\widetilde{\N}$ and each fiber is either an AF-algebra or a purely infinite simple $C^{*}$-algebra.  Here $\widetilde{\N} = \N \cup \{ \infty \}$ is the one-point compactification of $\N$.  We show that an ordered isomorphism between ideal-related $K$-theory with coefficients (as defined in \cite{mdrm:ethy}) lifts to an isomorphism between the stabilized $C^{*}$-algebras.  Moreover, if the $\infty$ fiber has torsion-free $K$-theory, then ideal-related $K$-theory with coefficients can be replace by ideal-related $K$-theory (which in general is a much simpler invariant).  This is done by proving a universal coefficient theorem involving ideal-related $K$-theory for $C^{*}$-algebras over $\widetilde{\N}$ whose $\infty$ fiber has torsion-free $K$-theory.  It was shown in \cite{mdrm:ethy} that a universal coefficient theorem involving ideal-related $K$-theory does not exist in general.  

We note that we can not use the results in \cite{ERRshift} since the extension 
\begin{align*}
0 \to \mathfrak{A}( \N ) \to \mathfrak{A} \to \mathfrak{A} (\infty) \to 0
\end{align*}    
does not satisfy the property that for every nonzero $a \in \mathfrak{A}( \infty )$, the ideal generated by $\tau ( a )$ in the corona algebra $\corona{ \mathfrak{A}(\N) }$ is $\corona{ \mathfrak{A}(\N)}$.  Here, $\tau$ denotes the Busby map of the above extension.  Instead, we prove existence and uniqueness theorems, which together with an intertwining argument give the desired result.

One of our motivations for studying this class of $C^{*}$-algebras is that this class contains the class of graph $C^{*}$-algebras whose primitive ideal space is $\widetilde{\N}$.  In fact, it was shown by the first named author in \cite{jg:t1graph} that a graph $C^\ast$-algebra with a $T_{1}$ (in particular Hausdorff) primitive ideal space has a canonical $C^{*}$-algebra over $\widetilde{\N}$ structure.  In this paper, we classify those for which this structure is tight over $\widetilde{\N}$  (see Definition~\ref{d:Xalgs}).  This paper contributes to the on going program to classify real rank zero graph $C^{*}$-algebras using ideal-related $K$-theory.  See \cite{ERR:gjogv} for an overview of the classification program for graph $C^{*}$-algebras with finitely many ideals.  To the authors knowledge, all known classification results for graph $C^{*}$-algebras involve graph $C^{*}$-algebras with finitely many gauge invariant ideals.   Thus, Theorem~\ref{t:classgraph} is the first classification result for graph $C^{*}$-algebras with infinitely many gauge-invariant ideals of mixed type.

\section{Preliminaries}

In this section, we recall the definition of $C^{*}$-algebras over $X$ and ideal-related $K$-theory (with coefficients) for $C^{*}$-algebras over a totally disconnected space $X$.  We also prove several structural properties of $C^{*}$-algebras over $\widetilde{\N}$ that will be used throughout the paper.  

Throughout the paper, $\Sigma \mathfrak{A}$ will denote the suspension $C_{0} ( \R ) \otimes \mathfrak{A}$ of $\mathfrak{A}$ and $\Sigma^{j} \mathfrak{A}$ is defined recursively $\Sigma ( \Sigma^{j-1} \mathfrak{A}) )$.

\subsection{$C^{*}$-algebras over topological spaces} Let $X$ be a topological space and let $\mathbb{O}( X)$ be the set of open subsets of $X$, partially ordered by set inclusion $\subseteq$.  A subset $Y$ of $X$ is called \emph{locally closed} if $Y = U \setminus V$ where $U, V \in \mathbb{O} ( X )$ and $V \subseteq U$.  The set of all locally closed subsets of $X$ will be denoted by $\mathbb{LC}(X)$.  For a $C^{*}$-algebra $\mathfrak{A}$, let $\mathbb{I} ( \mathfrak{A} )$ be the set of all closed two-sided ideals of $\mathfrak{A}$, partially ordered by $\subseteq$.  

\begin{defin}\label{d:Xalgs}
Let $\mathfrak{A}$ be a $C^{*}$-algebra.  Let $\mathrm{Prim} ( \mathfrak{A} )$ denote the \emph{primitive ideal space} of $\mathfrak{A}$, equipped with the usual hull-kernel topology, also called the Jacobson topology.

Let $X$ be a topological space.  A \emph{$C^{*}$-algebra over $X$} is a pair $( \mathfrak{A} , \psi )$ consisting of a $C^{*}$-algebra $\mathfrak{A}$ and a continuous map $\ftn{ \psi }{ \mathrm{Prim} ( \mathfrak{A} ) }{ X }$.  A $C^{*}$-algebra over $X$, $( \mathfrak{A} , \psi )$, is \emph{separable} if $\mathfrak{A}$ is a separable $C^{*}$-algebra.  We say that $( \mathfrak{A} , \psi )$ is \emph{tight} if $\psi$ is a homeomorphism.  
\end{defin}

We will always identify $\mathbb{O} ( \mathrm{Prim} ( \mathfrak{A} ) )$ and $\mathbb{I} ( \mathfrak{A} )$ using the canonical lattice isomorphism $U \mapsto \bigcap_{ \mathfrak{p} \in \mathrm{Prim} ( \mathfrak{A} ) \setminus U } \mathfrak{p}$.  Let $( \mathfrak{A} , \psi )$ be a $C^{*}$-algebra over $X$.  Then we get a map $\ftn{ \psi^{*} }{ \mathbb{O} ( X ) }{ \mathbb{O} ( \mathrm{Prim} ( \mathfrak{A} ) ) }$ defined by $U \mapsto \setof{ \mathfrak{p} \in \mathrm{Prim} ( \mathfrak{A} ) }{ \psi ( \mathfrak{p} ) \in U }$.  Using the lattice isomorphism $\mathbb{O} ( \mathrm{Prim} ( \mathfrak{A} ) ) \cong \mathbb{I} ( \mathfrak{A} )$, we get a map, which we again denote by $\psi^{*}$, from $\mathbb{O}(X)$ to $\mathbb{I}(\mathfrak{A} )$ by
\begin{align*}
U \mapsto \bigcap \setof{ \mathfrak{p} \in \mathrm{Prim} ( \mathfrak{A} ) }{ \psi ( \mathfrak{p} ) \notin U }.
\end{align*}
Denote this ideal by $\mathfrak{A}(U)$.  For $Y = U \setminus V \in \mathbb{LC} ( X )$, set $\mathfrak{A}(Y) = \mathfrak{A} (U) / \mathfrak{A}(V)$.   By \cite[Lemma~2.15]{rmrn:bootstrap}, $\mathfrak{A} ( Y)$ (up to a canonical choice of isomorphism) does not depend on $U$ and $V$.

By \cite[Lemma 2.25]{rmrn:bootstrap} it follows that if $X$ is a sober space (which is no actual restriction on $X$) then any $C^\ast$-algebra $\mathfrak A $ together with a map $\psi^\ast\colon \mathbb O(X) \to \mathbb I(A)$, which respects arbitrary suprema, finite infima and such that $\psi^\ast(\emptyset)=0,\psi^\ast(X)=\mathfrak A$, gives rise to a continuous map $\phi \colon \mathrm{Prim} ( \mathfrak A ) \to X$ such that $\psi^\ast=\phi^\ast$.

\begin{defin}
Let $(\mathfrak{A} , \psi )$ be a $C^{*}$-algebra over $X$.  We say that $( \mathfrak{A} , \psi )$ is \emph{continuous} if $\psi^{*}$ respects arbitrary infima, i.e., for any collection of open subsets $\{ U_{\lambda} \}$ of $X$, then 
\begin{align*}
\mathfrak{A} (U) = \bigcap_{\lambda} \mathfrak{A}(U_{\lambda})
\end{align*}  
where $U$ is the interior of $\bigcap_{\lambda}U_{\lambda}$.
\end{defin}

We should remark, that in the case that $X$ is a locally compact Hausdorff space, a $C^\ast$-algebra over $X$ is the same as a $C_{0}(X)$-algebra by \cite[Proposition 2.11]{rmrn:bootstrap}. Combining \cite[Lemma 2.9]{rmrn:bootstrap} with \cite[Corollary 2.2]{mn:bundles} one gets that continuous $C^\ast$-algebras over $X$ correspond exactly to continuous $C_{0}(X)$-algebras.

\begin{defin}
Let $\mathfrak{A}$ and $\mathfrak{B}$ be $C^{*}$-algebras over $X$.  A $*$-homomorphism $\ftn{ \phi }{ \mathfrak{A} }{ \mathfrak{B} }$ is \emph{$X$-equivariant} if $\phi ( \mathfrak{A} (U) ) \subseteq \mathfrak{B} ( U )$ for all $U \in \mathbb{O}(X)$.  Hence, for every $Y = U \setminus V$, $\phi$ induces a $*$-homomorphism $\ftn{ \phi_{Y} }{ \mathfrak{A} ( Y ) }{ \mathfrak{B} (Y) }$.  Let $\catc(X)$ be the category whose objects are $C^{*}$-algebras over $X$ and whose morphisms are $X$-equivariant $*$-homomorphisms.  

For any subspace $Y$ of $X$, we get a canonical covariant functor $i_Y \colon \catc(Y) \to \catc(X)$ given by $i_Y(\mathfrak A) (U) = \mathfrak A(Y\cap U)$, for every $U\in \mathbb O(X)$. In particular, if $Y=\{ x\}$ for a point $x\in X$ we obtain the functor $\ftn{ i_x }{ \catc \cong \catc(\{x\})}{ \catc(X) }$.
\end{defin}

Suppose that $X$ is a sober space. Let  $\mathfrak A_1 \xrightarrow{\lambda_{1,2}} \mathfrak A_2 \xrightarrow{\lambda_{2,3}} \dots$ be an inductive system with each $\mathfrak A_n$ a $C^\ast$-algebra over $X$ and $\lambda_{k,k+1}$ an $X$-equivariant $\ast$-homomorphism. We will say that $(\mathfrak A_k, \lambda_{k,k+1})$ is an inductive system of $C^\ast$-algebras over $X$. By exactness of the $C^\ast$-algebra inductive limit functor, the inductive limit $\mathfrak A$ is canonically a $C^\ast$-algebra over $X$ for which the induced $\ast$-homomorphisms $\iota_k \colon \mathfrak A_k \to \mathfrak A$ are $X$-equivariant.

\begin{lemma}\label{l:Xlimit}
Let $(\mathfrak A_k,\lambda_{k,k+1})$ and $(\mathfrak B_k,\mu_{k,k+1})$ be inductive systems of $C^\ast$-algebras over $X$ and let $\mathfrak A$ and $\mathfrak B$ be the respective inductive limits. Suppose that there are $X$-equivariant $\ast$-homomorphisms $\phi_n \colon \mathfrak A_n \to \mathfrak B_{k_n}$ which generate a $\ast$-homomorphism $\phi \colon \mathfrak A \to \mathfrak B$. Then $\phi$ is $X$-equivariant.
\end{lemma}
\begin{proof}
Let $U \in \mathbb O(X)$ and $a\in \mathfrak A(U)$, so that we should show that $\phi(a) \in \mathfrak B(U)$. Given $\epsilon >0$ we may, by the $X$-equivariant structure of $\mathfrak A$, find an $N$ and an $a'$ in $\mathfrak A_N(U)$ such that $\iota_N(a') \approx_{\epsilon } a$. Thus 
\[
\phi(a) \approx_{\epsilon} \phi ( \iota_N(a')) = \iota_{k_N}(\phi_N(a')) \in \mathfrak B(U),
\]
since both $\iota_{k_N}$ and $\phi_N$ are $X$-equivariant. Thus $\phi(a)\in \mathfrak B(U)$.
\end{proof}

\subsection{Invariants for $C^{*}$-algebras over a totally disconnected space.}

Let $\mathcal{P} \subseteq \N$ be the set consisting of~\(0\) and all prime powers. The relevance of the set $\mathcal{P}$ in the Universal Multicoefficient Theorem is that the groups $\Z_{p} := \Z /p \Z$ for $p\in \mathcal{P}$ are exactly the indecomposable abelian groups.

For non-zero $p\in \mathcal{P}$, let $\mathbb{I}_p$ be the mapping cone of the unital $*$-homomorphism $\C \to \mathsf{M}_p(\C)$.  For $p=0$, we let $\mathbb{I}_0 := \C$.  It is convenient to denote $\mathbb{I}_p$ by $\mathbb{I}_p^0$ and its suspension $\Sigma\mathbb{I}_p$ by $\mathbb{I}_p^1$.  Then for a $C^{*}$-algebra $\mathfrak{A}$:
\[
K_i(A;\Z_{p}) := \kk_i(\mathbb{I}_p,A)\cong \kk(\mathbb{I}_p^i,A), \quad i=0,1.
\]

Let us set $\mathbb{I} := \bigoplus_{p\in \mathcal{P} } \mathbb{I}_p$ and consider the ring $\kk_*(\mathbb{I},\mathbb{I})$ with multiplication given by the Kasparov product.  The non-unital subring
\[
\Lambda=\bigoplus_{p,q\in\mathcal{P}} \kk_*(\mathbb{I}_p,\mathbb{I}_q)
\]
of $\kk_*(\mathbb{I},\mathbb{I})$ is called the ring of \emph{B\"ockstein operations}.  It consists of matrices indexed by $\mathcal{P} \times \mathcal{P}$ with only finitely many non-zero entries $\lambda_{pq}\in \kk_*(\mathbb{I}_p,\mathbb{I}_q)$.  The Kasparov product
\[
\kk_*(\mathbb{I}_p,\mathbb{I}_q)\times \kk_*(\mathbb{I}_q, \mathfrak{A}) \to \kk_*(\mathbb{I}_p, \mathfrak{A})
\]
induces a natural $\Lambda$-module structure on the $\Z/2\times \mathcal{P}$-graded group
\[
\underline{K}( \mathfrak{A} )=\bigoplus_{p\in \mathcal{P}} K_*( \mathfrak{A};\Z_{p}).
\]

If $\mathfrak{A}$ is a separable $C^{*}$-algebra over a totally disconnected, metrizable, compact space~$X$, then $\underline{K}( \mathfrak{A} )$ has a natural structure of a module over the ring $C(X,\Lambda)$ of locally constant functions from $X$ to $\Lambda$,m and $K_{*} ( \mathfrak{A} )$ has a natural structure of a $\Z_2$-graded module over the ring $C(X,\Z)$ of locally constant functions from $X$ to $\Z$.  This is easily seen by observing that $ \mathfrak{A} \cong \bigoplus_{k=1}^n \mathfrak{A} (U_k)$ naturally for any clopen partition $(U_k)_{k=1}^n$ of $X$.  In the case when we have an evenly graded homomorphism of $\Z_2$-graded $C(X,\Z)$-modules, we will often abuse notation by just saying that we have a $C(X,\Z)$-homomorphism.

\subsection{Structural properties of $C^{*}$-algebras over $\widetilde{\N}$}  Let $\{ \mathfrak{A}_{n} \}_{n=1}^{\infty}$ be a sequence of $C^{*}$-algebras.  Set 

\begin{center}
$\displaystyle c_{0} ( \{ \mathfrak{A}_{n} \} ) = \setof{ \{ a_{n} \}_{ n = 1}^{ \infty } }{ a_{n} \in \mathfrak{A}_{n} , \ \lim_{ n \to \infty } \| a_{n} \| = 0 }$, 

\medskip

$\displaystyle \ell^{\infty} ( \{ \mathfrak{A}_{n} \} ) = \setof{ \{ a_{n} \}_{ n = 1}^{\infty } }{ \text{$a_{n} \in \mathfrak{A}_{n}$ and $\{ a_{n} \}_{ n = 1}^{\infty}$ is bounded}}$, 

\medskip

$\displaystyle q_{\infty} ( \{ \mathfrak{A}_{n} \} ) = \ell^{\infty} ( \{ \mathfrak{A}_{n} \} ) / c_{0} ( \{ \mathfrak{A}_{n} \} )$.
\end{center}
If $\{ \mathfrak{B}_{n} \}_{ n = 1}^{\infty}$ is a sequence of $C^{*}$-algebras and if $\{ \ftn{ \phi_{n} }{ \mathfrak{A}_{n} }{ \mathfrak{B}_{n} } \}_{ n = 1}^{\infty}$ is a sequence of $*$-homomorphisms, then $c_{0} ( \{ \phi_{n} \} )$ will denote the $*$-homomorphism from $c_{0} ( \{ \mathfrak{A}_{n} \} )$ to $c_{0} ( \{ \mathfrak{B}_{n} \} )$ induced by $\{ \phi_{n} \}_{ n = 1}^{\infty}$, $\ell^{\infty} ( \{ \phi_{n} \} )$ will denote the $*$-homomorphism from $\ell^{\infty} ( \{ \mathfrak{A}_{n} \} )$ to $\ell^{\infty} ( \{ \mathfrak{B}_{n} \} )$ induced by $\{ \phi_{n} \}_{ n = 1}^{\infty}$, and $q_{\infty} ( \{ \phi_{n} \} )$ will denote the $*$-homomorphism from $q_{\infty} ( \{ \mathfrak{A}_{n} \} )$ to $q_{\infty} ( \{ \mathfrak{B}_{n} \} )$ induced by $\{ \phi_{n} \}_{ n = 1}^{\infty}$. 

Let $\mathfrak{A}$ be a $C^{*}$-algebra over $\widetilde{\N}$.  For each $n \in \widetilde{\N}$, denote the $*$-homomorphism from $\mathfrak{A}$ to $\mathfrak{A}(n)$ by $\pi_{n}$.  The quotient map from $\ell^{\infty} ( \{ \mathfrak{A} (n) \} )$ to $q_{\infty} ( \{ \mathfrak{A}( n) \} )$ will be denoted $\rho_{\mathfrak{A}}$ or just $\rho$.  

\begin{lemma}\label{l:contsubalg}
Let $\mathfrak A$ be a continuous $C^\ast$-algebra over $\widetilde \N$, $\mathfrak B$ be a $C^\ast$-algebra and let $\iota \colon \mathfrak B \to \mathfrak A( \infty ) $ be an injective $*$-homomorphism. Construct the pullback diagram
\[
\xymatrix{
0 \ar[r] & \mathfrak A(\N) \ar[r] \ar@{=}[d] & \mathfrak E \ar[r] \ar[d] & \mathfrak B \ar[d]^\iota \ar[r] & 0 \\
0 \ar[r] & \mathfrak A(\N) \ar[r] & \mathfrak A \ar[r]^-{\pi_{\infty}} & \mathfrak A( \infty ) \ar[r] & 0.
}
\]
Then $\mathfrak E$ is a continuous $C^\ast$-algebra over $\widetilde \N$ when given the structure $\mathfrak E(U) = \mathfrak E\cap (\mathfrak A(U)\oplus i_\infty(\mathfrak B)(U) )$ for $U\subseteq \widetilde \N$ open.
\end{lemma}
\begin{proof}
That $\mathfrak E$ is a $C^\ast$-algebra over $\widetilde \N$ follows from \cite[Lemma~2.24]{mdrm:ethy} by observing that $\pi_{\infty} \colon \mathfrak A \to i_\infty(\mathfrak A(\infty))$ and $\iota \colon i_\infty ( \mathfrak B) \to i_\infty (\mathfrak A( \infty ))$ are $\widetilde \N$-equivariant. 

For continuity we check that if $U_1 \supsetneq U_2 \supsetneq \dots$ is a strictly decreasing sequence of open subsets of $\widetilde \N$, and we let $U$ denote the interior of $\bigcap_{n=1}^\infty U_n$, then $\bigcap_{n=1}^\infty \mathfrak E(U_n) = \mathfrak E(U)$. Observe that $U\subseteq \N$ since the sequence $\{ U_n \}_{n=1}^\infty$ is strictly decreasing, and thus $\mathfrak{A}(U) \subseteq \mathfrak{A}(\N)$. Since $\mathfrak A$ is continuous we have that
\[
\bigcap_{n=1}^\infty \mathfrak E(U_n) = \mathfrak{E}\cap \bigcap_{n=1}^\infty (\mathfrak A(U_n)\oplus  i_\infty(\mathfrak B)(U_{n}) ) = \mathfrak E \cap \left(\mathfrak A(U) \oplus \bigcap_{n=1}^\infty i_\infty(\mathfrak B)(U_{n}) \right).
\]
Since $\pi_{\infty} ( \mathfrak{A}(U)) = 0$ and $\iota$ is injective it follows that
\[
\bigcap_{n=1}^\infty \mathfrak E(U_n) = \mathfrak E \cap (\mathfrak A(U) \oplus 0) = \mathfrak E\cap (\mathfrak A(U) \oplus i_\infty(\mathfrak B)(U) ) = \mathfrak E(U).
\]
\end{proof}

\begin{lemma}\label{l:reducedBusby}
Let $\mathfrak A$ be a $C^\ast$-algebra over $\widetilde \N$, and consider the commutative diagram with exact rows
\[
\xymatrix{
0 \ar[r] & \mathfrak A(\N) \ar[r] \ar[d]^\cong & \mathfrak A \ar[r] \ar[d]^{ \ell^{\infty} ( \{ \pi_n \} ) } & \mathfrak A( \infty ) \ar[r] \ar[d]^{\overline \tau_{\mathfrak{A}}} & 0 \\
0 \ar[r] & c_{0} ( \{ \mathfrak A( n ) \} ) \ar@{=}[d] \ar[r] & \ell^\infty( \{ \mathfrak A( n ) \} ) \ar[r]_{\rho_{\mathfrak{A}}} \ar[d]^{\iota_{\mathfrak{A}}} & q_{\infty} ( \{ \mathfrak A( n ) \} ) \ar[r] \ar[d]^{\overline \iota_{ \mathfrak{A} }} & 0 \\
0 \ar[r] & c_{0} (\{ \mathfrak A( n ) \} ) \ar[r] & \mathcal{M} ( c_{0} (  \{ \mathfrak A( n ) \} ) )\ar[r] & \mathcal{Q} ( c_{0}( \{ \mathfrak A( n ) \} )) \ar[r] & 0,
}
\]
where $\pi_n \colon \mathfrak A \to \mathfrak A ( n )$ are the canonical epimorphisms, $\iota$ is the canonical inclusion and $\overline \tau_{\mathfrak{A}}$ and $\overline \iota$ are the unique induced $*$-homomorphisms. Let $\tau_{\mathfrak{A}}$ denote the Busby map of the top row. Then $\tau_{\mathfrak{A}} = \overline \iota \circ \overline \tau_{\mathfrak{A}}$.  Also, 
\begin{align*}
\mathfrak{A} \cong \mathfrak{A} (  \infty  ) \oplus_{ \overline{\tau}_{\mathfrak{A}} , \rho_{\mathfrak{A}} } \ell^{\infty} ( \{ \mathfrak{A}(n) \} ) 
\end{align*}
via the $*$-isomorphism $a \mapsto ( \pi_{\infty}(a), \{ \pi_{n} ( a ) \}_{ n = 1}^{\infty} )$.
\end{lemma}

\begin{proof}
That $\mathfrak{A} \cong \mathfrak{A} (  \infty  ) \oplus_{ \overline{\tau}_{\mathfrak{A}} , \rho_{\mathfrak{A}} } \ell^{\infty} ( \{ \mathfrak{A}(n) \} ) $ follows by a diagram chase in the top part of the diagram.  Let $\sigma : \mathfrak A \to \mathcal M(\mathfrak A(\N)) \cong \ell^\infty ( \{ \mathcal M(\mathfrak A( n )) \} )$ denote the map induced by the Busby map. Note that if $a\in \mathfrak A$ and $b\in \mathfrak A( n )$, then $ab = \pi_n(a)b\in \mathfrak A( n )$. Hence if $\{ a_n \}_{ n = 1}^{\infty} \in c_{0} ( \{ \mathfrak A( n ) \} )$ we get that
\[
\sigma (a) \left( \{a_n\}_{n=1}^\infty \right)= \left\{ \pi_n(a) a_n \right\}_{n=1}^\infty = (\iota \circ \ell^{\infty} ( \{ \pi_n\} ) (a)) \left( \{ a_n\}_{n=1}^\infty \right)
\]
and similarly $\{ a_n\}_{n=1}^\infty \sigma (a) = \{a_n\}_{n=1}^\infty (\iota \circ \ell^{\infty} ( \{ \pi_n\} ) (a)$ and thus $\sigma = \iota \circ \ell^{\infty}(\{ \pi_n\} ) $. This implies that $\tau_{\mathfrak{A}} = \overline \iota \circ \overline \tau_{\mathfrak{A}}$ by a canonical uniqueness argument.
\end{proof}

\begin{defin}
Let $\mathfrak{A}$ be a $C^{*}$-algebra over $\widetilde{\N}$.  The $*$-homomorphism $\overline{\tau}_{\mathfrak{A}}$ in Lemma~\ref{l:reducedBusby}, will be called the \emph{reduced Busby map of $\mathfrak{A}$}.    
\end{defin}

To ease the notation throughout the paper, we will remove the subscript $\mathfrak{A}$ of all $*$-homomorphisms in Lemma~\ref{l:reducedBusby} whenever we are working with only one algebra $\mathfrak{A}$.

\begin{lemma}\label{l:continuous}
Let $\mathfrak A$ be a $C^\ast$-algebra over $\widetilde \N$ and let $\ftn{ \overline \tau }{\mathfrak A( \infty)  }{ q_{\infty} ( \{ \mathfrak A( n ) \} ) }$ be the reduced Busby map. Then $\mathfrak A$ is continuous if and only if for any non-zero $a\in \mathfrak A(  \infty )$ and every lift $\{ a_n \}_{n=1}^\infty \in \ell^{\infty}( \{ \mathfrak A( n ) \} )$ of $\overline \tau(a)$, there is an $\epsilon > 0$ such that $\| a_n \| < \epsilon$ for only finitely many $n\in \N$.
\end{lemma}
\begin{proof}
Suppose that $a\in \mathfrak A( \infty ) \setminus \{ 0 \}$, $\{ a_n \}_{n=1}^\infty \in \ell^\infty ( \{ \mathfrak A( n ) \} )$ lifts $\overline \tau(a)$ such that $\| a_n\| < 1/m$ for infinitely many $n\in \N$ for every $m\in \N$. Then we may pick an infinite subsequence $\{ a_{n_k} \}_{k=1}^\infty$ such that $\lim_{k\to \infty}\| a_{n_k} \| =0$. Let $F=\{ n_1 , n_2,\dots \}$ and define
\[
b_n = \left\{ \begin{array}{ll} a_n & \text{ if } n\notin F \\ 0 & \text{ if } n\in F. \end{array} \right.  
\]
Then $\{ b_n \}_{n=1}^\infty$ is a lift of $\overline \tau(a)$. Set $U_k = \widetilde \N \setminus \{ n_1,\dots n_k\}$. Then $\bigcap_{k=1}^\infty U_k = \widetilde \N \setminus F$ and since $F$ is infinite, the interior of this set, say $U$, is a subset of $\N$. Identify $\mathfrak A$ with the pullback $\mathfrak A( \infty ) \oplus_{\overline \tau,\rho} \ell^{\infty} (\{ \mathfrak A( n ) \} )$. Then the element $\left(a, \{ b_n\}_{n=1}^\infty\right)$ is in $\bigcap_{k=1}^\infty \mathfrak A(U_k)$ and since $a\neq 0$ this ideal is not contained in $\mathfrak A(\N)$. In particular, this implies that $\bigcap_{k=1}^\infty \mathfrak A(U_k) \neq \mathfrak A(U)$ and thus $\mathfrak A$ is not continuous.

For the other implication suppose that $\mathfrak A$ is not continuous and let $U_1 \supsetneq U_2 \supsetneq \dots$ be a (strictly decreasing) sequence of open subsets of $\widetilde \N$ such that $\bigcap_{k=1}^\infty \mathfrak A(U_k)  \neq \mathfrak{A} (U)$, where $U$ is the interior of $\bigcap_{ k = 1}^{\infty} U_{k}$.  Note that $\bigcap_{k=1}^\infty \mathfrak A(U_k) \nsubseteq \mathfrak A(\N)$ otherwise $\bigcap_{k=1}^\infty \mathfrak A(U_k)  = \mathfrak{A} (U)$.  Identifying $\mathfrak A$ with the pullback $\mathfrak A( \infty ) \oplus_{\overline \tau,\rho} \ell^\infty( \{ \mathfrak A( n ) \} )$, we may pick an element $\left(a, \{ a_n \}_{n=1}^\infty \right)$ in $\bigcap_{k=1}^\infty \mathfrak A(U_k)$ such that $a$ is non-zero. Now $\{ a_n \}_{n=1}^\infty$ is a lift of $\overline \tau(a)$ and $\| a_n\|=0$ for $n\in \N \setminus \bigcap_{k=1}^\infty U_k$ which is an infinite set.
\end{proof}

\begin{corol}\label{c:continuous}
Let $\mathfrak A$ be a continuous $C^\ast$-algebra over $\widetilde \N$ and let $\ftn{ \overline \tau }{ \mathfrak A(\infty) }{ \ell^\infty ( \{ \mathfrak A( n ) \} ) }$ denote the reduced Busby map. Suppose that $p\in \mathfrak A( \infty )$ is a non-zero projection and that $\{ q_n \}_{n=1}^\infty \in \ell^\infty( \{ \mathfrak A(n ) \} )$ is a projection which lifts $\overline \tau(a)$. Then $q_n=0$ for only finitely many $n\in \N$.
\end{corol}

\begin{proof}
This follows from Lemma~\ref{l:continuous} since each $q_n$ is a projection for each $n\in \N$ and thus has norm $0$ or $1$.
\end{proof}

\begin{corol}\label{c:tight}
Let $\mathfrak{A}$ be a $C^{*}$-algebra over $\widetilde{\N}$ such that $\mathfrak{A}(n)$ is non-zero and simple for all $n \in \widetilde{\N}$.  Then $\mathfrak{A}$ is a tight $C^{*}$-algebra over $\widetilde{\N}$ if and only if $\mathfrak{A}$ is a continuous $C^{*}$-algebra over $\widetilde{\N}$.
\end{corol}

\begin{proof}
Since $\mathfrak{A}$ is a tight $C^{*}$-algebra over $\widetilde{\N}$, we have that the map $\ftn{ \psi }{ \mathrm{Prim} (\mathfrak{A} ) }{ \widetilde{\N} }$ is a homeomorphism.  Hence, $\psi$ is an open map which implies that $\mathfrak{A}$ is a continuous $C^{*}$-algebra over $\widetilde{\N}$.  

Suppose $\mathfrak{A}$ is a continuous $C^{*}$-algebra over $\widetilde{\N}$.  Let $\mathfrak{I}$ be an ideal of $\mathfrak{A}$.  Suppose $\mathfrak{I} \subseteq \mathfrak{A}(\N)$.  Since $\mathfrak{A}(n)$ is simple for all $n \in \N$, we have that $\mathfrak{I} = \mathfrak{A}(U)$ for some $U \subseteq \N$.  Suppose $\mathfrak{I}$ is not a subset of $\mathfrak{A}(\N)$.  By Lemma~\ref{l:continuous}, 
\begin{align*}
F = \setof{ n \in \N }{ \mathfrak{I} \cap \mathfrak{A}(n)  = 0 }
\end{align*}  
is finite.  Set $U = \widetilde{ \N } \setminus F$.  Then $U$ is an open subset of $\widetilde{\N}$ and $\mathfrak{I} \subseteq \mathfrak{A}(U)$.  Let $\ftn{ \iota }{ \mathfrak{I} }{ \mathfrak{A}(U) }$ be the inclusion map.  Then the diagram
\begin{align*}
\xymatrix{
0 \ar[r] & \mathfrak{A}(U \setminus \{ \infty \} ) \ar[r] \ar@{=}[d] & \mathfrak{I} \ar[r] \ar[d]^{\iota} & \mathfrak{A}(\infty) \ar[r] \ar@{=}[d] & 0 \\
0 \ar[r] & \mathfrak{A}(U \setminus \{ \infty \} ) \ar[r] & \mathfrak{A}(U) \ar[r] & \mathfrak{A}(\infty) \ar[r] & 0
}
\end{align*}
is commutative and the rows are exact.  Thus, $\iota$ is surjective which implies that $\mathfrak{I} = \mathfrak{A}(U)$.  We have just shown that the lattice map $\mathbb O(\widetilde{\N}) \to \mathbb I(\mathfrak A)$ is surjective, thus it remains to show that it is injective. Let $U,V\in \mathbb O(\widetilde{\N})$ such that $\mathfrak A (U) = \mathfrak A(V)$. If $U\subseteq \N$ then as $\mathfrak A(U)$ is tight over $U$ it follows that $V=U$. If $\infty \in U$ then $F= \widetilde{\N} \setminus U$ is a finite subset of $\N$ and $\mathfrak A \cong \mathfrak A(F) \oplus \mathfrak A(U)$ naturally. It follows that $F = \widetilde{\N} \setminus V$ and thus $U=V$.
\end{proof}

\begin{defin}
A $C^{*}$-algebra $\mathfrak{A}$ has \emph{weak cancellation} if any pair of projections $p$ and $q$ in $\mathfrak{A}$ that generate the same closed ideal $\mathfrak{I}$ in $\mathfrak{A}$ and have the same image in $K_{0} ( \mathfrak{I} )$ must be Murray-von Neumann equivalent.  If $\mathsf{M}_{n} ( \mathfrak{A} )$ has weak cancellation for every $n$, then we say that $\mathfrak{A}$ has \emph{stable weak cancellation}.
\end{defin}

Note that $\mathfrak{A}$ has stable weak cancellation if and only if $\mathfrak{A} \otimes \K$ has weak cancellation.  Ara, Moreno, and Pardo in \cite{amp:nonstablekthy} showed that every graph $C^{*}$-algebra has stable weak cancellation.  It is an open question if every real rank zero $C^{*}$-algebra has stable weak cancellation.

\begin{propo}\label{p:rrstwc}
Let $\mathfrak{A}$ be a $C^{*}$-algebra over $\widetilde{\N}$.  If $\mathfrak{A}$ has real rank zero and $\mathfrak{A}(n)$ has stable weak cancellation for all $n \in \widetilde{\N}$, then $\mathfrak{A}$ has stable weak cancellation.  
\end{propo}

\begin{proof}
Since $\mathfrak{A}(n)$ has stable weak cancellation for each $n \in \N$, then $c_{0} ( \{ \mathfrak{A}(n) \} )$ has stable weak cancellation.  The proposition now follows from \cite[Lemma~3.15]{err:fullext}.
\end{proof}

\begin{lemma}\label{l:liftingfdquot}
Let $\mathfrak{A}$ be a $C^{*}$-algebra and let $\mathfrak{I}$ be an ideal of $\mathfrak{A}$ such that $\mathfrak{A} / \mathfrak{I}$ is a finite dimensional $C^{*}$-algebra.  If for every projection $p\in \mathfrak A$ the corner $p\mathfrak{I}p$ has an approximate identity consisting of projections, and every projection in $\mathfrak{A} / \mathfrak{I}$ lifts to a projection in $\mathfrak{A}$, then there exists a $*$-homomorphism $\ftn{ \phi }{ \mathfrak{A} / \mathfrak{I} }{ \mathfrak{A} }$ such that $\pi \circ \phi = \id_{ \mathfrak{A} / \mathfrak{I} }$, where $\ftn{ \pi }{ \mathfrak{A} }{ \mathfrak{A} / \mathfrak{I}}$ is the quotient map.

Consequently, if $\mathfrak{A}$ is a $C^{*}$-algebra with an ideal $\mathfrak{I}$ such that for every projection $p\in \mathfrak A$ the corner $p\mathfrak{I}p$ has an approximate identity consisting of projections, $\mathfrak{A} / \mathfrak{I}$ is an AF-algebra, and every projection in $\mathfrak{A}/\mathfrak{I}$ lifts to a projection in $\mathfrak{A}$, then there exists a sequence of finite dimensional sub-$C^{*}$-algebras $\{ \mathfrak{C}_{k} \}_{k = 1}^{\infty}$ of $\mathfrak{A}$ such that $\mathfrak{C}_{k} \cap \mathfrak{I} = 0$ for all $k$, $\mathfrak{C}_{k} + \mathfrak{I} \subseteq \mathfrak{C}_{k+1} + \mathfrak{I}$ for all $k$, and $\bigcup_{ k = 1}^{\infty} ( \mathfrak{C}_{k} +\mathfrak{I} )$ is dense in $\mathfrak{A}$.
\end{lemma}

\begin{proof}
The first part of the lemma is proved in the same way as \cite[Lemma~9.8]{dim}.  The following are the key ingredients of the proof: (1) the existence of an approximate identity consisting of projections for $p \mathfrak{I} p$ for every projection $p \in \mathfrak{A}$ and (2) every projection in $\mathfrak{A}/ \mathfrak{I}$ lifts to a projection in $\mathfrak{A}$.

Suppose $\mathfrak{A}$ is a $C^{*}$-algebra with an ideal $\mathfrak{I}$ such that $p \mathfrak{I} p$ has an approximate identity consisting of projections for all projections $p \in \mathfrak{A}$, $\mathfrak{A} / \mathfrak{I}$ is an AF-algebra, every projection in $\mathfrak{A}/\mathfrak{I}$ lifts to a projection in $\mathfrak{A}$.  Since $\mathfrak{A} / \mathfrak{I}$ is an AF-algebra, there exists an increasing sequence of finite dimensional sub-$C^{*}$-algebras $\{ \mathfrak{D}_{k} \}_{ k = 1}^{\infty}$ of $\mathfrak{A} / \mathfrak{I}$ such that $\mathfrak{A} / \mathfrak{I} = \overline{ \bigcup_{ k = 1}^{\infty } \mathfrak{D}_{k} }$.  By the first part of the lemma, we have a sequence of $*$-homomorphisms, $\{ \ftn{ \phi_{k} }{ \mathfrak{D}_{k} }{ \mathfrak{A} }  \}_{ k = 1}^{\infty}$ such that $\pi \circ \phi_{k} = \id_{ \mathfrak{D}_{k}}$.  

Set $\mathfrak{C}_{k} = \phi_{k}( \mathfrak{D} _{k} )$.  Then $\mathfrak{C}_{k}$ is a finite dimensional sub-$C^{*}$-algebra of $\mathfrak{A}$.  Since $\pi \circ \phi_{k} = \id_{ \mathfrak{D}_{k}}$, we have that $\mathfrak{C}_{k} \cap \mathfrak{I} = 0$.  Note that $\mathfrak{C}_{k} + \mathfrak{I} \subseteq \mathfrak{C}_{k+1} + \mathfrak{I}$ since $\mathfrak{D}_{k} \subseteq \mathfrak{D}_{k+1}$.  Let $x \in \mathfrak{A}$ and let $\epsilon > 0$.  Since $\pi ( \bigcup_{ k =1}^{\infty} \mathfrak{C}_{k} ) = \bigcup_{ k = 1}^{\infty} \mathfrak{D}_{k}$, there exists $y_{1} \in \mathfrak{C}_{k}$ (for some $k$) such that  $\| \pi(x) - \pi( y_{1} ) \| < \epsilon$.  Thus, there exists $y_{2} \in \mathfrak{I}$ such that $\| x - y_{1} - y_{2} \| < \epsilon$.  Since $y_{1} + y_{2} \in \bigcup_{ k = 1}^{\infty} ( \mathfrak{C}_{k} + \mathfrak{I} )$, we have just shown that $\bigcup_{ k = 1}^{\infty} ( \mathfrak{C}_{k} +\mathfrak{I} )$ is dense in $\mathfrak{A}$.
\end{proof}

\begin{defin}
An extension $0 \to \mathfrak{B} \overset{\iota}{\to}  \mathfrak{E} \overset{\pi}{\to} \mathfrak{A} \to 0$ is said to be \emph{quasi-diagonal} if $\mathfrak{B}$ has an approximate identity consisting of projections $\{ p_{n} \}_{n=1}^{\infty}$ such that 
\begin{align*}
\lim_{ n \to \infty } \| \iota ( p_{n} ) x - x \iota( p_{n} ) \| = 0
\end{align*}
for all $x \in \mathfrak{E}$.
\end{defin}

We end this section by showing that the extension $0 \to \mathfrak{A}(\N) \to \mathfrak{A} \to \mathfrak{A}(\infty) \to 0$ is a quasi-diagonal extension under mild assumptions on the fibers.  This fact will be used repeatedly throughout the paper.

\begin{propo}\label{p:quasidiag}
Let $\mathfrak A$ be a separable $C^\ast$-algebra over $\widetilde{\mathbb N}$ such that each $\mathfrak A(n)$ has real rank zero and $\mathfrak A(\infty)$ an AF-algebra. Then the extension
\[
0 \to \mathfrak A(\mathbb N) \to \mathfrak A \to \mathfrak A(\infty) \to 0
\]
is a quasi-diagonal extension.
\end{propo}

\begin{proof}
A functional calculus argument implies that every projection in $q_{\infty} ( \{ \mathfrak{A}(n) \} )$ lifts to a projection in $\ell^{\infty}( \{ \mathfrak{A}(n) \} )$.  By Lemma~\ref{l:reducedBusby}, we have that $\mathfrak{A} \cong \mathfrak{A}( \infty ) \oplus _{ \overline{\tau} , \rho } \ell^{ \infty } ( \{ \mathfrak{A}(n) \} )$.  It is now clear that every projection in $\mathfrak{A} (\infty)$ lifts to a projection in $\mathfrak{A}$.  By Lemma~\ref{l:liftingfdquot}, there exists a sequence of finite dimensional sub-$C^{*}$-algebras of $\mathfrak A$, $\{ \mathfrak B_k\}_{k=1}^\infty$, such that $\mathfrak B_k \cap \mathfrak A(\mathbb N)=0$ and $\bigcup_{ k  =1}^{\infty} ( \mathfrak{B}_{k} + \mathfrak{A}(\N) )$ is dense in $\mathfrak{A}$. By the epimorphisms onto the direct summands $\pi_n\colon \mathfrak A \to \mathfrak A(n)$ for $n\in \mathbb N$, we get a $*$-homomorphism $\overline \sigma \colon \mathfrak A \to \ell^\infty( \{ \mathfrak A(n)\})$. Let $\{ p_m^n \}_{m=1}^{\infty}$ be an increasing approximate identity of projections in $\mathfrak A(n)$ for each $n\in \mathbb N$. By passing to subsequences we may assume that
\[
\| p_m^n x - xp_m^n \| < 1/m \text{ for } x\in \pi_n \left(\bigcup_{k=1}^m(\mathfrak B_k)_1\right), n\leq m,
\]
since the closed unit balls $(\mathfrak B_k)_1$ are compact. Define $p_m := \sum_{n=1}^m p_m^n$. Clearly $\{ p_m\}_{m=1}^\infty$ is an approximate identity of projections in $\mathfrak A(\mathbb N)$. We claim that this is quasi-central in $\mathfrak A$. Identifying $\mathfrak A$ with the pull-back $\mathfrak A(\infty)  \oplus_{ \overline \tau, \rho }  \ell^\infty(\{ \mathfrak A(n)\})$, we see that if $x\in \mathfrak A(\mathbb N)$ and $a\in \mathfrak A$, then $xa = x \overline \sigma (a)$ and $ax = \overline \sigma(a) x$. Let $a\in \mathfrak A$. We should show that $\lim_{ m \to \infty } \| p_ma - ap_m \| = 0$.  Let $\epsilon > 0$.  For some large $N$ we may pick $x\in \bigcup_{k=1}^N \mathfrak B_k$ and $y\in \mathfrak A(\mathbb N)$ such that $\| a  - x - y \| < \frac{ \epsilon }{ 4 }$.     

Since $\lim_{ m \to \infty } \| p_my - yp_m \|  = 0$, there exists $N_{2} \geq N$ such that $\| p_{m} y - y p_{m} \| < \frac{ \epsilon }{ 4 }$ for all $m \geq N_2$.  For $m\geq N$ we have
\begin{eqnarray*}
\| p_m x - x p_m \| &=& \| p_m \overline{\sigma}(x) - \overline \sigma(x) p_m \| \\
&=& \max_{n=1,\dots, m} \| p_m^n \pi_n(x) - \pi_n(x) p_m^n\| \leq \|x\|/m.
\end{eqnarray*}
Suppose $m \geq \max \left\{ N_{2} , \frac{ 4 ( \| x \| + 1 ) }{ \epsilon } \right\}$.  Then
\begin{align*}
\| p_{m} a - a p_{m} \| &\leq \| p_{m} ( a - x - y )  \| + \| p_{m} ( x + y ) - ( x + y ) p_{m} \| + \| ( x+y- a ) p_{m} \| \\
				&< \epsilon.	
\end{align*}
Hence, $\lim_{ m \to \infty } \| p_{m} a - a p_{m} \| = 0$.
\end{proof}

\section{Uniqueness Theorem}

In this section, we show that two ``full'' $\widetilde{\N}$-equivariant $*$-homomorphisms from $\mathfrak{A}$ to $\mathfrak{B}$ are approximately unitarily equivalent provided that they agree on ideal-related $K$-theory with coefficient.  Theorem~\ref{t:uniq} will be our key uniqueness result which allows us to use an approximate intertwining argument.  We will also use Theorem~\ref{t:uniq} to lift isomorphisms between ideal-related $K$-theory with coefficient to a ``full'' $\widetilde{\N}$-equivariant $*$-homomorphism (see Theorem~\ref{t:existence}).   

\begin{defin}
An element $a$ in a $C^{*}$-algebra $\mathfrak{A}$ is said to be \emph{full} if the ideal generated by $a$ is $\mathfrak{A}$. 
\begin{itemize}
\item[(1)] Let $\mathfrak{A}$ and $\mathfrak{B}$ be tight $C^{*}$-algebras over $X$.  An $X$-equivariant $*$-homomorphism $\ftn{ \phi }{ \mathfrak{A} }{ \mathfrak{B} }$ is said to be a \emph{full $X$-equivariant $*$-homomorphism} if for all $Y \in \mathbb{LC}(X)$, we have that $\phi_{Y} (a)$ is full in $\mathfrak{B}(Y)$ whenever $a$ is full in $\mathfrak{A}(Y)$.

\item[(2)] Let $\mathfrak{A}$ and $\mathfrak{B}$ be $C^{*}$-algebras.  A $*$-homomorphism $\ftn{ \phi }{ \mathfrak{A} }{ \mathfrak{B}}$ is said to be \emph{full} if for every nonzero $a \in \mathfrak{A}$, we have that $\phi(a)$ is full in $\mathfrak{B}$.
\end{itemize}
\end{defin}

\begin{defin}
Let $\mathfrak{A}$ and $\mathfrak{B}$ be separable $C^{*}$-algebras over $X$.  Two $X$-equivariant $*$-homomorphisms $\ftn{ \phi , \psi }{ \mathfrak{A} }{ \mathfrak{B} }$ are said to be \emph{approximately unitarily equivalent} if there exists a sequence of unitaries $\{ u_{n} \}_{ n = 1}^{ \infty }$ in $\multialg{ \mathfrak{B} }$ such that 
\begin{align*}
\lim_{ n \to \infty } \| u_{n} \phi ( a ) u_{n}^{*} - \psi ( a ) \| = 0 
\end{align*}
for all $a \in \mathfrak{A}$.
\end{defin}

\begin{defin}\label{d:class-uniq-hom}
We will be interested in classes of $C^{*}$-algebras $\mathcal{C}$ satisfying the following property: if $\mathfrak{A}, \mathfrak{B} \in \mathcal{C}$ and $\ftn{ \phi, \psi }{ \mathfrak{A}  }{ \mathfrak{B} \otimes \K }$ are full $*$-homomorphisms such that $\underline{K} ( \phi ) = \underline{K}(\psi)$, then for each non-zero projection $e$ in $\mathfrak{A} $, there exists a sequence of partial isometries $\{ v_{n} \}_{ n = 1}^{\infty}$ in $\mathfrak{B} \otimes \K$ such that $v_{n}^{*} v_{n} = \phi ( e)$, $v_{n} v_{n}^{*} = \psi(e)$, and 
\begin{align*}
\lim_{ n \to \infty } \norm{ v_{n} \phi ( x ) v_{n}^{*} - \psi(x) } = 0
\end{align*}
for all $x \in e \mathfrak{A} e$.
\end{defin}

The class of simple AF-algebras and the class of separable, nuclear, purely infinite simple $C^{*}$-algebras in the bootstrap category $\mathcal{N}$ satisfy the properties of Definition~\ref{d:class-uniq-hom}. This follows e.g. by \cite[Theorem 2.6]{lin:fullexts}.

\begin{theor}\label{t:uniq}
For each $n \in \N$, let $\mathcal{C}_{n}$ be a class of $C^{*}$-algebras satisfying the properties of Definition~\ref{d:class-uniq-hom}.  Let $\mathfrak{A}_{1}$ be a separable $C^{*}$-algebra over $\widetilde{\N}$ with real rank zero such that $\mathfrak{A}_{1} ( \infty )$ is an AF-algebra and $\mathfrak{A}_{1} ( n ) \in \mathcal{C}_{n}$ for each $n \in \N$.  Let $\mathfrak{A}_{2}$ be a separable $C^{*}$-algebra over $\widetilde{\N}$ such that $\mathfrak{A}_{2}$ is a stable $C^{*}$-algebra, and $\mathfrak{A}_{2} (n) \in \mathcal{C}_{n}$ for each $n \in \N$.

If $\ftn{ \phi, \psi }{ \mathfrak{A}_{1} }{ \mathfrak{A}_{2} }$ are $\widetilde{\N}$-equivariant $*$-homomorphisms such that $\underline{K} ( \phi_{ n } ) = \underline{K}( \psi_{n} )$ for all $n \in \N$, $\phi_{n}$ and $\psi_{n}$ are full $*$-homomorphisms, and for each projection $e \in \mathfrak{A}_{1}$, we have that $\phi(e)$ and $\psi(e)$ are Murray-von Neumann equivalent, then $\phi$ and $\psi$ are approximately unitarily equivalent.
\end{theor}

\begin{proof}
By Proposition~\ref{p:quasidiag}, we have that $0 \to \mathfrak{A}_{1} ( \N ) \to \mathfrak{A}_{1} \to \mathfrak{A}_{1} ( \infty ) \to 0$ is a quasi-diagonal extension.  Let $\{ e_{k} \}_{ k = 1}^{\infty}$ be an approximate identity consisting of projections of $\mathfrak{A}_{1}(\N)$ such that 
\begin{align*}
\lim_{ n \to \infty } \norm{ e_{k} x - x e_{k} } = 0
\end{align*}
for all $x \in \mathfrak{A}_{1}$.

Since $\mathfrak{A}_{1}(  \infty )$ is an AF-algebra and $\mathfrak{A}_{1}$ has real rank zero, by Lemma~\ref{l:liftingfdquot} there exists a sequence of finite dimensional sub-$C^{*}$-algebras $\{ \mathfrak{B}_{k} \}_{ k = 1}^{ \infty }$ such that $\mathfrak{B}_{k} \cap \mathfrak{A}_{1} ( \N  ) = \{ 0 \}$, $\mathfrak{B}_{k} + \mathfrak{A}_{1} (\N) \subseteq \mathfrak{B}_{k+1} + \mathfrak{A}_{1} (\N)$, and $\bigcup_{ k  =1}^{\infty} ( \mathfrak{B}_{k} + \mathfrak{A}_{1} ( \N ) )$ is dense in $\mathfrak{A}_{1}$.  Let $\epsilon > 0$ and $\mathcal{F}$ be a finite subset of $\mathfrak{A}_{1}$ so that we should find a unitary $u\in \multialg{\mathfrak A_2}$ for which $\| u \phi(a) u^\ast - \psi(a)\| < \epsilon$ for all $a\in \mathcal F$.  Since $\mathfrak{B}_{k} + \mathfrak{A}_{1} (\N) \subseteq \mathfrak{B}_{k+1} + \mathfrak{A}_{1} (\N)$ and $\bigcup_{ k  =1}^{\infty} ( \mathfrak{B}_{k} + \mathfrak{A}_{1} (\N) )$ is dense in $\mathfrak{A}_{1}$, we may assume that there exist $m \in \N$ and a finite subset $\mathcal{G}$ of $\mathfrak{A}_{1} ( \N )$ such that every element of $\mathcal{F}$ is of the form $y_{1} + y_{2}$ where $y_{1}$ is a generator of $\mathfrak{B}_{m}$ and $y_{2} \in \mathcal{G}$.  Since $\mathfrak{B}_{m}$ is a finite dimensional $C^{*}$-algebra (hence semprojective),
\begin{align*}
\lim_{ k \to \infty } \| e_{k} x - x e_{k} \| = 0
\end{align*}
for all $x \in \mathfrak{A}_{1}$, and $\{ e_{k} \}_{k \in \N }$ is an approximate identity for $\mathfrak{A}_{1} ( \N )$ consisting of projections, there exist $k \in \N$, a finite dimensional sub-$C^{*}$-algebra $\mathfrak{D}$ of $\mathfrak{A}_{1}$ with $\mathfrak{D} \subseteq ( 1_{ \multialg{ \mathfrak{A}_{1} } } - e_{k} ) \mathfrak{A}_{1} ( 1_{ \multialg{ \mathfrak{A}_{1} } } - e_{k} )$ and $\mathfrak{D} \cap \mathfrak{A}_{1} ( \N ) = \{ 0 \}$, and there exists a finite subset $\mathcal{H}$ of $e_{k} \mathfrak{A}_{1} ( \N ) e_{k}$ such that for all $x \in \mathcal{F}$, there exist $y_{1} \in \mathfrak{D}$ and $y_{2} \in \mathcal{H}$
\begin{align*}
\| x - ( y_{1} + y_{2} ) \| < \frac{ \epsilon }{ 3 }.
\end{align*} 

Set $\mathfrak{D} = \bigoplus_{ \ell = 1}^{s } \mathsf{M}_{n_{\ell}}$ and let $\{ f_{ij}^{\ell} \}_{ i , j = 1}^{ n_{ \ell }}$ be a system of matrix units for $\mathsf{M}_{n_{ \ell } }$.  By assumption, $\phi ( f_{11}^{\ell} )$ is Murray-von Neumann equivalent to $\psi ( f_{11}^{ \ell } )$.  Hence, there exists $v_{ \ell } \in \mathfrak{A}_{2}$ such that $v_{\ell}^{*} v_{\ell} = \phi ( f_{11}^{\ell} )$ and $v_{\ell} v_{\ell}^{*} = \psi ( f_{11}^{ \ell } )$.  Set 
\begin{align*}
u_{1} = \sum_{ \ell = 1}^{k} \sum_{ i =1}^{n_{ \ell } } \psi ( f_{i1}^{\ell} ) v_{\ell} \phi ( f_{1i}^{\ell} )
\end{align*}  
Then, $u_{1}$ is a partial isometry in $\mathfrak{A}_{1}$ such that $u_{1}^{*} u_{1} = \phi ( 1_{ \mathfrak{D} } )$, $u_{1} u_{1}^{*} = \psi ( 1_{ \mathfrak{D} } )$, and $u_{1} \phi ( x ) u_{1}^{*} = \psi (x)$ for all $x \in \mathfrak{D}$.

Since $e_{k}$ is a projection in $\mathfrak{A}_{1} ( \N )$, we have that $e_{k} = \bigoplus_{ n \in U } e_{k, n}$ for some finite subset $U \subseteq \N$ and $e_{k, n} \neq 0$.  Choose finite subsets $\mathcal{H}_{n}$ of $e_{k, n} \mathfrak{A}_{1} (  n  ) e_{k,n}$ such that $\mathcal{H} \subseteq \bigoplus_{ n \in U } \mathcal{H}_{n}$.  Since $\phi$ and $\psi$ are $\widetilde{\N}$-equivariant $*$-homomorphisms, we have that $\phi_{U} = \bigoplus_{ n \in U } \phi_{n}$ and $\psi_{U} = \bigoplus_{ n \in U } \psi_{ n }$.  By assumption, we have that  $\ftn{ \phi_{ n } , \psi_{ n } }{ \mathfrak{A}_{1} (  n  ) }{ \mathfrak{A}_{2} ( n ) }$ are full $*$-homomorphisms.  Let $\beta_{n}$ be the inclusion of $e_{k,n} \mathfrak{A}_{1} (n) e_{k,n}$ into $\mathfrak{A}_{1}(n)$.  Since $\underline{K} ( \phi_{n} ) = \underline{K}( \psi_{n} )$, we have that $\underline{K}( \phi_{n} \circ \beta_{n} ) = \underline{K} ( \psi_{ n  } \circ \beta_{n})$.  Since $\mathfrak{A}_{1} ( n)$ and $\mathfrak{A}_{2} (n)$ are elements of $\mathcal{C}_{n}$, there exists a partial isometry $v_{n} \in \mathfrak{A}_{2}(n)$ such that $v_{n}^{*} v_{n} = \phi_{ n  }( e_{k,n})$, $v_{n}v_{n}^{*} = \psi_{  n  } ( e_{k,n})$, and 
\begin{align*}
\norm{ v_{n} ( \phi_{n} \circ \beta_{n} )( x ) v_{n}^{*} - ( \psi_{  n  } \circ \beta_{n}) ( x ) } < \frac{ \epsilon }{ 3 }
\end{align*}
for all $x \in \mathcal{H}_{n}$.  Set $u_{2} = \bigoplus_{ n \in U } v_{n}$.  Since $U$ is finite, $u_{2}$ is a partial isometry in $\mathfrak{A}_{2} ( \N )$.  Moreover, $u_{2}^{*}u_{2} = \phi( e_{k} )$, $u_{2}u_{2}^{*} = \psi( e_{k} )$, and 
\begin{align*}
\norm{ u_{2} \phi ( x ) u_{2}^{*} - \psi ( x ) } < \frac{ \epsilon }{ 3 }
\end{align*} 
for all $x \in \mathcal{H}$.  Since $\mathfrak{A}_{2}$ is separable and stable, there exists $u_{3} \in \multialg{ \mathfrak{A}_{2} }$ such that $u_{3}^{*} u_{3} = 1_{ \multialg{ \mathfrak{A}_{2} } } - ( u_{1} + u_{2} )^{*} ( u_{1} + u_{2} )$ and $u_{3} u_{3}^{*} = 1_{ \multialg{ \mathfrak{A}_{2} } } - ( u_{1} + u_{2} )( u_{1} + u_{2} )^{*}$.  Set $u = u_{1} + u_{2} + u_{3} \in \multialg{ \mathfrak{A}_{2} }$.  Then $u$ is a unitary in $\multialg{ \mathfrak{A}_{2} }$.  

Let $x \in \mathcal{F}$.  Choose $y_{1} \in \mathfrak{D}$ and $y_{2} \in \mathcal{H}$ such that $\| x - ( y_{ 1 } + y_{2} ) \| < \frac{ \epsilon }{3} $.  Then
\begin{align*}
\| u \phi ( x ) u^{*} - \psi ( x ) \| &\leq \| u \phi (x) u^{*} - u \phi( y_{1} + y_{2} ) u^{*} \|   \\
				&\qquad + \| u_{1} \phi ( y_{1} ) u_{1} + u_{2} \phi ( y_{2} ) u_{2}^{*} - \psi ( y_{1} ) - \psi ( y_{2} ) \| \\
					&\qquad \qquad  + \| \psi ( y_{1} + y_{2} ) - \psi ( x ) \|   \\
						&< \epsilon.
\end{align*}
It now follows that $\phi$ and $\psi$ are approximately unitarily equivalent since $\mathfrak{A}_{1}$ is separable. 
\end{proof}

\section{Existence Theorem}

\subsection{Asymptotic morphisms}  In this section, we define equivariant $E$-theory as in \cite{mdrm:ethy}.  From now on, let $T = [ 0, \infty )$, $C^{b} ( T , \mathfrak{A} )$ be the $C^{*}$-algebra of all bounded continuous functions from $T$ to $\mathfrak{A}$, and $C_{0} ( T , \mathfrak{A} )$ be the $C^{*}$-algebra of all continuous functions from $T$ to $\mathfrak{A}$ which vanish at $\infty$. 

\begin{defin}
Let $\mathfrak{A}$ and $\mathfrak{B}$ be $C^{*}$-algebras.  An \emph{asymptotic morphism} from $\mathfrak{A}$ to $\mathfrak{B}$ is a map $\ftn{ \phi = ( \phi_{t} )_{t\in T} }{ \mathfrak{A} }{ C^{b} ( T , \mathfrak{B} ) }$ such that the composition
\begin{align*}
\mathfrak{A} \overset{\phi}{\rightarrow} C^{b}( T , \mathfrak{B} ) \twoheadrightarrow \mathfrak{B}_{\infty} := C^{b} ( T , \mathfrak{B} ) / C_{0} ( T , \mathfrak{B} )
\end{align*}
is a $*$-homomorphism.

Suppose $\mathfrak{A}$ and $\mathfrak{B}$ are $C^{*}$-algebras over $X$.  
\begin{itemize}
\item[(1)] An asymptotic morphism from $\mathfrak{A}$ to $\mathfrak{B}$, $( \phi_{t} )_{t\in T}$, is said to be \emph{approximately $X$-equivariant} if for each open set $U$ of $X$
\begin{align*}
\lim_{ t \to \infty } \| \phi_{t}(a) \|_{ X \setminus U } = 0
\end{align*}
for all $a \in \mathfrak{A}(U)$, where $\| b \|_{ X \setminus U }$ is the norm of $b$ in the quotient $\mathfrak{B} ( X ) / \mathfrak{B} (U)$.

\item[(2)] Two asymptotic morphism $\phi_{0}$ and $\phi_{1}$ from $\mathfrak{A}$ to $\mathfrak{B}$ that are approximately $X$-equivariant are said to be \emph{homotopic} if there exists an asymptotic morphism $\Phi$ from $\mathfrak{A}$ to $C([0,1] , \mathfrak{B} )$ that is approximately $X$-equivariant, $\mathrm{ev}_{0} \circ \Phi = \phi_{0}$, and $\mathrm{ev}_{1} \circ \Phi = \phi_{1}$.  The set of homotopy classes of approximately $X$-equivariant asymptotic morphisms from $\mathfrak{A}$ to $\mathfrak{B}$ will be denoted by $[[ \mathfrak{A} , \mathfrak{B} ]]_{X}$.
\end{itemize}
\end{defin}

Recall, that $\Sigma \mathfrak A = C_{0}( \mathbb R) \otimes \mathfrak A$.

\begin{defin}
Let $X$ be a second countable sober space and let $\mathfrak{A}$ and $\mathfrak{B}$ be $C^{*}$-algebras over $X$.  Define 
\begin{align*}
E_{0} ( X ; \mathfrak{A} , \mathfrak{B} ) = [[ \Sigma \mathfrak{A} \otimes \K , \Sigma \mathfrak{B} \otimes \K ]]_{X} \quad \text{and} \quad E_{1} ( X ; \mathfrak{A} ,  \mathfrak{B} ) = E_{0} ( X; \mathfrak{A} , \Sigma \mathfrak{B} ).
\end{align*}
Equipped with the Cuntz sum these sets are abelian groups.
\end{defin}

Let $X$ be totally disconnected, metrizable, compact space.  If $\gamma \in E_{0} ( X ; \mathfrak{A} , \mathfrak{B} )$, then $\gamma$ induces a $C(X, \Lambda )$-homomorphism $\ftn{ \underline{K} ( \gamma ) }{ \underline{K}( \mathfrak{A} ) }{ \underline{ K } ( \mathfrak{B} ) }$ and a $C(X, \Z )$-homomorphism $\ftn{ K_{*} ( \gamma ) }{ K_{*}( \mathfrak{A} ) }{ K_{*} ( \mathfrak{B} ) }$.

\subsection{Lifting isomorphisms of ideal-related $K$-theory with coefficients}

The following lemma shows that two elements in $E( \widetilde{\N} , \mathfrak{A} , \mathfrak{B} )$ induce the same $C( \widetilde{\N} , \Z )$-homomorphism provided that the induced elements in $E( \mathfrak{A}(n) , \mathfrak{B} (n) )$ are equal for all $n \in \widetilde{\N}$.

\begin{lemma}\label{l:ethyequal}
Let $\mathfrak{A}$ and $\mathfrak{B}$ be separable, nuclear $C^{*}$-algebras over $\widetilde{\N}$ and $\alpha, \beta \in E( \widetilde{\N} , \mathfrak{A} , \mathfrak{B} )$.  Suppose that $\alpha_{n} = \beta_{n}$ in $E(\mathfrak A(n) , \mathfrak B(n) )$ for all $n \in \widetilde{\N}$.  Then the $C( \widetilde{\N}, \Z )$-homomorphisms $K_{*} ( \alpha )$ and $K_{*} ( \beta )$ are equal.  
\end{lemma}

\begin{proof}
Note that the diagram 
\begin{align*}
\xymatrix{
E( \widetilde{\N} ; \iota_{\N} ( \mathfrak{A} (\N) ) , \Sigma\iota_{\N} ( \mathfrak{B} ( \N ) ) ) \ar[r] \ar[d] & E( \widetilde{\N} ; \iota_{\N} ( \mathfrak{A} (\N) ) , \Sigma \mathfrak{B}  ) \ar[r] \ar[d] & E( \widetilde{\N} ;\iota_{\N} ( \mathfrak{A} (\N) ) , \Sigma \iota_{\infty} ( \mathfrak{B} (\infty )  ) )  \ar[d] \\
E( \widetilde{\N} ; \iota_{\infty} ( \mathfrak{A}( \infty ) ) , \iota_{\N} ( \mathfrak{B} ( \N ) ) ) \ar[r] \ar[d] & E( \widetilde{\N} ; \iota_{\infty} ( \mathfrak{A}( \infty ) ) , \mathfrak{B}  ) \ar[r] \ar[d] & E( \widetilde{\N} ; \iota_{\infty} ( \mathfrak{A}( \infty ) ) , \iota_{\infty} ( \mathfrak{B} (\infty )  ) ) \ar[d] \\ 
E( \widetilde{\N} ; \mathfrak{A} , \iota_{\N} ( \mathfrak{B} ( \N ) ) ) \ar[r] \ar[d] & E( \widetilde{\N} ; \mathfrak{A} , \mathfrak{B}  ) \ar[r] \ar[d] & E( \widetilde{\N} ;\mathfrak{A} , \iota_{\infty} ( \mathfrak{B} (\infty )  ) ) \ar[d] \\
E( \widetilde{\N} ; \iota_{\N} ( \mathfrak{A} (\N) ) , \iota_{\N} ( \mathfrak{B} ( \N ) ) ) \ar[r] & E( \widetilde{\N} ; \iota_{\N} ( \mathfrak{A} (\N) ) , \mathfrak{B}  ) \ar[r] & E( \widetilde{\N} ;\iota_{\N} ( \mathfrak{A} (\N) ) , \iota_{\infty} ( \mathfrak{B} (\infty )  ) ) 
}
\end{align*}
is commutative, and the rows and columns are exact sequences. 

Since $\alpha_{n} = \beta_{n}$ for all $n \in \N$, we have that $\alpha - \beta$ is in the image of the homomorphism from $E( \widetilde{\N} , \iota_{\infty} ( \mathfrak{A}( \infty ) ) , \mathfrak{B}  )$ to $E( \widetilde{\N} , \mathfrak{A} , \mathfrak{B}  )$.  Let $y$ be an element in $E( \widetilde{\N} , \iota_{\infty} ( \mathfrak{A}( \infty ) ) , \mathfrak{B}  )$  which is mapped to $\alpha - \beta$.  

Since $\iota_{\N} ( \mathfrak{A} ( \N ) )$ is a continuous $C^{*}$-algebra over $\widetilde{\N}$, by \cite[Theorem~5.4]{mdrm:ethy}, we have that $\kk ( \widetilde{\N} , \iota_{\N} ( \mathfrak{A} ( \N ) ),  \mathfrak{C} ) \cong E( \widetilde{\N} , \iota_{\N} ( \mathfrak{A} ( \N ) ), \mathfrak{C} )$ for any separable $C^{*}$-algebra $\mathfrak{C}$ over $\widetilde{\N}$.  Let $\mathfrak{D}$ be a $C^{*}$-algebra.  Then, by \cite[Proposition~3.12]{rmrn:bootstrap}, 
\begin{align*}
\kk( \widetilde{\N} , \iota_{\N} ( \mathfrak{A} ( \N ) ) , \iota_{\infty}( \mathfrak{D} ) ) \cong \kk ( \N , \mathfrak{A} ( \N ) , r_{ \widetilde{\N} }^{ \N } ( \iota_{\infty}( \mathfrak{D} )  ) ) = \kk ( \N , \mathfrak{A} ( \N ), 0 ) = 0.
\end{align*}  
Hence, $E( \widetilde{\N} ,\iota_{\N} ( \mathfrak{A} (\N) ) , \iota_{\infty} ( \mathfrak{B} (\infty )  ) ) = E( \widetilde{\N} ,\iota_{\N} ( \mathfrak{A} (\N) ) , \Sigma \iota_{\infty} ( \mathfrak{B} (\infty )  ) ) = 0$.  This implies that the homomorphism from $E( \widetilde{\N} ,\iota_{\infty} ( \mathfrak{A}(\infty) ) , \iota_{\infty} ( \mathfrak{B} (\infty )  ) )$ to $E( \widetilde{\N} , \mathfrak{A}  , \iota_{\infty} ( \mathfrak{B} (\infty )  ) )$ is an isomorphism.  Since $\alpha_\infty = \beta_{\infty}$, we have that $y$ is in the image of the homomorphism from $E( \widetilde{\N} , \iota_{\infty} ( \mathfrak{A}( \infty ) ) , \iota_{\N} ( \mathfrak{B} ( \N ) ) )$ to $E( \widetilde{\N} , \iota_{\infty} ( \mathfrak{A}( \infty ) ) , \mathfrak{B}  )$.        

Let $z$ be a lifting of $y$.   Note that $\Hom_{ C( \widetilde{\N} , \Z ) } ( K_{*} ( \iota_{ \infty } ( \mathfrak{A} (\infty) ) ), K_{*} (\iota_{\N} ( \mathfrak{B} ( \N ) ) )) = 0$.  Hence, $K_{*}(z)$ is zero which implies that the homomorphism on $K$-theory induced by $y$ is zero.  Therefore, $K_{*} ( \alpha ) - K_{*} ( \beta ) = 0$.
\end{proof}

\begin{lemma}\label{l:unitaryequiv}
Let $\mathfrak{A}$ be a finite dimensional $C^{*}$-algebra and let $\{ \mathfrak{B}_{n} \}_{n = 1}^{\infty}$ be a sequence of separable, stable $C^{*}$-algebras such that each $\mathfrak{B}_{n}$ has weak cancellation.  Suppose that $\ftn{ \phi , \psi }{ \mathfrak{A} }{ q_{\infty} ( \{ \mathfrak{B}_{n} \} ) }$ are $*$-homomorphisms such that 
\begin{itemize}
\item[(1)] $K_{0} ( \phi ) = K_{0}( \psi )$;

\item[(2)] for each nonzero projection $p \in \mathfrak{A}$, there exist a projection $q = \{ q_{n} \}_{ n  =1}^{\infty} \in \ell^{\infty} ( \{ \mathfrak{B}_{n} \} )$ and an $N \in \N$ such that $q_{n}$ is full in $\mathfrak{B}_{n}$ for all $n \geq N$ and $\rho ( q ) = \phi ( p )$; and 

\item[(3)] for each nonzero projection $p \in \mathfrak{A}$, there exist a projection $q = \{ q_{n} \}_{ n  =1}^{\infty} \in \ell^{\infty} ( \{ \mathfrak{B}_{n} \} )$ and an $N \in \N$ such that $q_{n}$ is full in $\mathfrak{B}_{n}$ for all $n \geq N$ and $\rho ( q ) = \psi ( p )$.
\end{itemize}
Then there exist $*$-homomorphisms $\ftn{\widetilde{\phi}_{n} , \widetilde{\psi}_{n} }{ \mathfrak{A} }{ \mathfrak{B}_{n} }$ such that $\ell^{\infty} ( \{ \widetilde{\phi}_{n} \} )$ and $\ell^{\infty} ( \{ \widetilde{\psi}_{n} \} )$ are liftings of $\phi$ and $\psi$ respectively and there exists a unitary $u = \{ u_{n} \}_{ n =1}^{\infty}$ in $\ell^{\infty} (\{  \multialg{ \mathfrak{B}_{n} } \} )$ such that
$u_{n} \widetilde{ \phi }_{n} (a) u_{n}^{*} = \widetilde{ \psi }_{n} (a)$ for all $a \in \mathfrak{A}$.
\end{lemma}

\begin{proof}
By (2) we may assume that $\mathfrak B_{n}$ has a full projection for each $n$. Hence $\ell^\infty(\{ \mathfrak B_{n}\})$ has an approximate identity of projection, since each $\mathfrak B_{n}$ is separable. Thus since each $\mathfrak{B}_{n}$ is stable, we have that $K_{0} ( q_{\infty} ( \{ \mathfrak{B}_{n} \} ) ) \cong \frac{ \prod_{ n = 1}^{\infty} K_{0} ( \mathfrak{B}_{n} ) }{ \bigoplus_{n = 1}^{\infty} K_{0} ( \mathfrak{B}_{n} ) }$ where the isomorphism is induced by the coordinate projections.  Using this identification, the fact that $\mathfrak{A}$ is finite dimensional, and assumptions (1), (2), (3), there exist $*$-homomorphisms $\ftn{ \widetilde{\phi}, \widetilde{\psi} }{ \mathfrak{A} }{ \ell^{\infty} ( \{ \mathfrak{B}_{n} \} ) }$ and there exists $N \in \N$ such that $\rho \circ \widetilde{ \phi } = \phi$, $\rho  \circ \widetilde{\psi} = \psi$, $K_{0} ( \widetilde{ \phi } ) = K_{0} ( \widetilde{ \psi } )$, and for all $n \geq N$ and for every nonzero projection $p \in \mathfrak{A}$, we have that the $n$th coordinate of $\widetilde{\phi}(p)$ and $\widetilde{\psi}(p)$ are full projections in $\mathfrak{B}_{n}$.

Note that $\widetilde{ \phi } = \ell^{\infty }( \{ \widetilde{ \phi }_{n} \} )$ and $\widetilde{ \psi } = \ell^{\infty}( \{ \widetilde{ \psi }_{n} \} )$, where $\ftn{ \widetilde{ \phi }_{n} , \widetilde{ \psi }_{n} }{ \mathfrak{A} }{ \mathfrak{B}_{n} }$ are $*$-homomorphisms.  By construction, for each nonzero projection $p \in \mathfrak{A}$, we have that $\widetilde{ \phi }_{n} (p)$ and $\widetilde{ \psi }_{ n } (p)$ are full projections in $\mathfrak{B}_{n}$ for all $n \geq N$ and $K_{0} ( \widetilde{ \phi }_{n} ) = K_{0} ( \widetilde{ \phi }_{n} )$.  

Let $\mathfrak{A} = \bigoplus_{ k = 1}^{m} \mathsf{M}_{n(k)}$.  Let $\{ e_{ij}^{k} \}$ be a system of matrix units for $\mathsf{M}_{n(k)}$.  Let $n \geq N$.  Since $\widetilde{ \phi }_{n} ( e_{11}^{k} )$ and $\widetilde{ \psi }_{ n } ( e_{11}^{k} )$ are full projections in $\mathfrak{B}_{n}$, $[ \widetilde{ \phi }_{n} ( e_{11}^{k} ) ] = [ \widetilde{ \psi }_{n} ( e_{11}^{k} ) ]$, and $\mathfrak{B}_{n}$ is a stable $C^{*}$-algebra with weak cancellation, we have that there exists $v_{n,k} \in \mathfrak{B}_{n}$ such that $v_{n,k}^{*} v_{n,k} =\widetilde{ \phi }_{n} ( e_{11}^{k} )$ and $v_{n,k} v_{n,k}^{*} =\widetilde{ \psi }_{n} ( e_{11}^{k} )$.  

Set 
\begin{align*}
v_{n} = \sum_{ k = 1}^{m} \sum_{ i = 1}^{n(k)} \widetilde{ \psi }_{n} (  e_{i1}^{k} ) v_{n,k} \widetilde{\phi}_{n} ( e_{1i}^{k} ).
\end{align*}  
Then $v_{n}$ is a partial isometry in $\mathfrak{B}_{n}$ such that $v_{n}^{*} v_{n} = \widetilde{ \phi }_{n} ( 1_{ \mathfrak{A} } )$, $v_{n} v_{n}^{*} = \widetilde{ \psi }_{n} ( 1_{ \mathfrak{A} } )$, and 
\begin{align*}
v_{n} \widetilde{\phi}_{n} (x) v_{n}^{*} = \widetilde{\psi}_{n} (x)
\end{align*}
for all $x \in \mathfrak{A}$.  Since $\mathfrak{B}_{n}$ is a separable, stable $C^{*}$-algebra, there exists a partial isometry $w_{n} \in \multialg{ \mathfrak{B}_{n} }$ such that $w_{n}^{*} w_{n} = 1_{ \multialg{ \mathfrak{B}_{n} } } - \widetilde{\phi}_{n} ( 1_{\mathfrak{A}} )$ and $w_{n}w_{n}^{*} = 1_{ \multialg{ \mathfrak{B}_{n} } } - \widetilde{\psi}_{n} ( 1_{\mathfrak{A}} )$.  Then $u_{n} = v_{n} + w_{n}$ is a unitary in $\multialg{ \mathfrak{B}_{n} }$ such that 
\begin{align*}
\mathrm{Ad} ( u_{n} ) \circ \widetilde{\phi}_{n}  = \widetilde{\psi}_{n}.
\end{align*}
Set $u_{n} = 1_{ \multialg{\mathfrak{B}_{n} } }$ and redefining $\widetilde{\phi}_{n} = \widetilde{ \psi }_{n} = 0$ for $1 \leq n < N$, we get the desired result.
\end{proof}

\begin{lemma}\label{l:fullelement}
Let $\mathfrak{A}$ be a tight $C^{*}$-algebra over $\widetilde{\N}$.  For each open subset $U \subseteq \widetilde{\N}$, we have that $a \in \mathfrak{A}(U)$ is full if and only if $\pi_{n} ( a ) \neq 0$ for all $n \in U$.
\end{lemma}

\begin{proof}
By Lemma \ref{l:reducedBusby} we may assume that $\mathfrak{A} = \mathfrak{A} ( \infty ) \oplus_{ \overline{\tau} , \rho } \ell^{ \infty } ( \{ \mathfrak{A} (n) \} )$, where $\overline{\tau}$ is the reduced Busby map. Let $U$ be an open subset of $\widetilde{\N}$.  If $\infty \notin U$, then  
\begin{align*}
\mathfrak{A}(U) = \setof{ ( 0 , \{ x_{n} \}_{n = 1}^{\infty} ) \in \mathfrak{A} }{ \text{$x_{n} = 0$ if $n \notin U$}}.
\end{align*}
If $\infty \in U$, then  
\begin{align*}
\mathfrak{A}(U) = \setof{ ( x_\infty , \{ x_{n} \}_{n = 1}^{\infty} ) \in \mathfrak{A} }{ \text{$x_\infty \in \mathfrak{A}(\infty)$ and $x_{n} = 0$ if $n \notin U$}}.
\end{align*}
It is now clear that $(x_{\infty} , \{ x_{n} \}_{ n = 1}^{\infty} ) \in \mathfrak{A}(U)$ is full if and only if $x_{n} \neq 0$ for all $n \in U$.  

\end{proof}

\begin{lemma}\label{l:fullhomo}
Let $\mathfrak{A}$ be a continuous $C^{*}$-algebra over $\widetilde{\N}$ and let $\mathfrak{B}$ be a tight $C^{*}$-algebra over $\widetilde{\N}$.  Suppose for each $n \in \N$, there exist an injective $*$-homomorphism $\ftn{ \phi_{n} }{ \mathfrak{A} ( n ) }{ \mathfrak{B} ( n ) }$, a unitary $u_{n} \in \multialg{ \mathfrak{B} (n) }$, and there exists an injective $*$-homomorphism $\ftn{ \phi_{\infty} }{ \mathfrak{A} ( \infty ) }{ \mathfrak{B} ( \infty ) }$ such that 
\begin{align*}
 q_{\infty} (\{ \mathrm{Ad}( u_{n} ) \circ \phi_{n} \} ) \circ \overline{\tau}_{  \mathfrak{A}  } =  \overline{ \tau }_{ \mathfrak{B}  } \circ \phi_{\infty} 
\end{align*}
and $\overline{\tau}_{ \mathfrak{A}  }$ and $\overline{\tau}_{ \mathfrak{B}  }$ are the reduced Busby maps of $\mathfrak{A}$ and $\mathfrak{B}$ respectively.  Then there exists an $\widetilde{\N}$-equivariant $*$-homomorphism $\ftn{ \psi }{ \mathfrak{A} }{ \mathfrak{B} }$ such that $E( \phi_{n} ) = E( \psi_{n} )$ in $E( \mathfrak{A} ( n ) , \mathfrak{B}(n) )$ for all $n \in \widetilde{\N}$ and if $a \in \mathfrak{A}$ with $F = \setof{ n \in \widetilde{\N} }{ \pi_{n} ( a ) = 0 }$, then $\psi(a)$ is full in $\mathfrak{B} ( \widetilde{\N} \setminus F )$. 
\end{lemma}

\begin{proof}
By Lemma \ref{l:reducedBusby} we may assume that $\mathfrak{A} = \mathfrak{A} ( \infty ) \oplus_{ \overline{\tau}_{ \mathfrak{A}  } , \rho_{ \mathfrak{A} } } \ell^{ \infty } ( \{ \mathfrak{A} (n) \} )$ and that $\mathfrak{B} = \mathfrak{B} ( \infty ) \oplus_{  \overline{\tau}_{ \mathfrak{B}  } , \rho_{ \mathfrak{B} } } \ell^{ \infty } ( \{ \mathfrak{B} (n) \} )$.  Set $\widetilde{\phi}_{\N} =\ell^{\infty} ( \{ \mathrm{Ad}( u_{n} ) \circ \phi_{n}  \} )$.  Since 
\begin{align*}
 q_{\infty} (\{ \mathrm{Ad}( u_{n} ) \circ \phi_{n} \} ) \circ \overline{\tau}_{  \mathfrak{A}  } =  \overline{ \tau }_{ \mathfrak{B}  } \circ \phi_{\infty} 
\end{align*} is precisely the pull-back relation, $\ftn{ \psi }{ \mathfrak{A} }{ \mathfrak{B} }$ by $\psi ( ( a, x ) ) = ( \phi_{\infty}(a), \widetilde{\phi}_{\N}(x) )$ is a well-defined $*$-homomorphism.  A computation shows that $\psi$ is $\widetilde{\N}$-equivariant since $u_{n}$ is a unitary.

Let $(a,x)\in \mathfrak A$ and let $F= \setof{ n \in \widetilde{\N}}{\pi_n(a,x) = 0}$. Then $U = \widetilde{\N} \setminus F$ is open by Lemma \ref{l:continuous} and $(a,x) \in \mathfrak A(U)$. Since $\psi$ is $\widetilde{\N}$-equivariant, $\psi(a,x) \in \mathfrak B(U)$. Since $u$ is a unitary in $\ell^{\infty} ( \{ \multialg{ \mathfrak{B}(n) } \} )$ and $\phi_{n}$ is injective for each $n \in \widetilde{\N}$, we have that $F = \setof{ n \in \widetilde{\N} }{ \pi_{n} (\psi(a,x)) = 0 }$.  Therefore, by Lemma~\ref{l:fullelement}, $\psi(a,x)$ is full in $\mathfrak{B} (U)$ since $\mathfrak{B}$ is a tight $C^{*}$-algebra over $\widetilde{\N}$.  

By the construction of $\psi$, we have that  $E( \phi_{n} ) = E( \psi_{n} )$ in $E( \mathfrak{A} ( n ) , \mathfrak{B}(n) )$ for all $n \in \widetilde{\N}$.
\end{proof}

We will use the following observation several times for the rest of the paper, without further mentioning: If $\mathfrak A$ is a $C^\ast$-algebra over $\widetilde{\N}$ for which $\mathfrak A(n)$ has real rank zero for each $n\in \widetilde{\N}$, then $\mathfrak A$ has real rank zero. This follows since $\mathfrak A(\N)$ and $\mathfrak A(\infty)$ have real rank zero and the map $K_0(\mathfrak A)\to K_0(\mathfrak A(\infty) )$ is surjective.

Recall that a \emph{Kirchberg algebra} is a separable, nuclear, purely infinite simple $C^{*}$-algebra.  Let $\mathcal{N}$ be the bootstrap category defined in \cite{UCT}.  
 
\begin{theor}\label{t:existence}
Let $\mathfrak{A}$ and $\mathfrak{B}$ be tight, stable $C^{*}$-algebras over $\widetilde{\N}$.  Suppose for each $n \in \N$, that $\mathfrak{A}(n)$ is an AF-algebra or a Kirchberg algebra in $\mathcal{N}$ and that $\mathfrak{B}(n)$ is an AF-algebra or a Kirchberg algebra in $\mathcal{N}$, and suppose that $\mathfrak{A}( \infty )$ and $\mathfrak{B}(\infty)$ are AF-algebras. 

If $\gamma \in E( \widetilde{\N} , \mathfrak{A}, \mathfrak{B} )$ is invertible such that $K_{0} ( \gamma_{n} )$ is an order isomorphism for all $n \in \widetilde{ \N }$, then there exists a full $\widetilde{ \N }$-equivariant $*$-homomorphism $\ftn{ \phi }{ \mathfrak{A}   }{ \mathfrak{B} }$ such that $\underline{K} ( \phi_{ n } ) = \underline{K} ( \gamma_{n} )$ for all $n \in \widetilde{\N}$ and such that the $C( \widetilde{\N} , \Z )$-homomorphisms $K_{*} ( \phi )$ and $K_{*} ( \gamma )$ are equal. 
\end{theor}

\begin{proof}
Every AF-algebra and every Kirchberg algebra have stable weak cancellation, so by Proposition~\ref{p:rrstwc}, $\mathfrak{A}$ and $\mathfrak{B}$ have weak cancellation.  Since $K_{0} ( \mathfrak{C} )_{+} = K_{0} ( \mathfrak{C} )$ for any Kirchberg algebra $\mathfrak{C}$ and $K_{0} ( \mathfrak{D} )_{+} \neq K_{0} ( \mathfrak{D} )$ for any non-zero AF-algebra $\mathfrak{D}$, we get for each $n \in \widetilde{\N}$ that either $\mathfrak{A} ( n )$ and $\mathfrak{B} ( n )$ are both AF-algebras or both Kirchberg algebras.  Therefore by the classification of AF-algebras \cite{af} and the Kirchberg-Phillips classification (\cite{kirchpure} and \cite{phillipspureinf}), there exists a $*$-isomorphism $\ftn{ \phi_{n} }{ \mathfrak{A}( n ) }{ \mathfrak{B} (n) }$ such that $E ( \phi_{n} ) = \gamma_{n}$ in $E( \mathfrak{A}(n), \mathfrak{B} (n) )$.  Define $\phi_{U} = \bigoplus_{ n \in U } \phi_{n}$ for all $U \subseteq \N$.  Then $\phi_{U}$ is a $*$-isomorphism from $\mathfrak{A} (U)$ to $\mathfrak{B} (U)$ and $E( \phi_{U} ) = \gamma_{U}$ in $E ( \mathfrak{A} (U) , \mathfrak{B} (U) )$.

Set $\widetilde{\phi}_{\N} = \ell^{ \infty }( \{ \phi_{n} \} )$ and set $\overline{\phi}_{\N} = q_{\infty} ( \{ \phi_{n} \} )$.  Then $\ftn{ \widetilde{ \phi }_{ \N } }{ \ell^{\infty} ( \{ \mathfrak{A} ( n ) \} ) }{ \ell^{\infty} ( \{ \mathfrak{B} ( n )\} ) }$ and $\ftn{ \overline{ \phi }_{ \N } }{ q_{ \infty} ( \{ \mathfrak{A}(n) \} ) }{  q_{ \infty} ( \{ \mathfrak{B}(n) \} ) }$ are $*$-isomorphisms.  Since $\mathfrak{A} ( \infty )$ is an AF-algebra, there exists a sequence of finite dimensional sub-$C^{*}$-algebras of $\mathfrak{A}(\infty)$, $\{ \mathfrak{F}_{n} \}_{ n = 1}^{\infty}$, such that $\mathfrak{F}_{n} \subseteq \mathfrak{F}_{n+1}$ and $\bigcup_{ n = 1}^{\infty} \mathfrak{F}_{n}$ is dense in $\mathfrak{A}(\infty)$.  Let $\mathfrak{D}_{k}$ be the pullback of the diagram
\begin{align*}
\xymatrix{
		&  \iota_{\infty} ( \mathfrak{F}_{k} ) \ar[d] \\
\mathfrak{A} \ar[r] & \iota_{\infty} ( \mathfrak{A}( \infty ) ) .		
}
\end{align*}
Then, for each $k \in \N$, $\mathfrak{D}_{k}$ is a $C^{*}$-algebra over $\widetilde{\N}$ by \cite[Lemma 2.24]{mdrm:ethy}, and there exist $\widetilde{ \N }$-equivariant $*$-homomorphisms $\ftn{ \iota_{k} }{ \mathfrak{D}_{k} }{ \mathfrak{A} }$ and $\ftn{ \lambda_{k, k+1} }{ \mathfrak{D}_{k} }{ \mathfrak{D}_{k+1} }$ such that $\mathfrak{A} = \varinjlim ( \mathfrak{D}_{k} , \lambda_{k, k+1} )$ and the diagram 
\begin{align*}
\xymatrix{
0 \ar[r] & \mathfrak{A}( \N ) \ar[r] \ar@{=}[d] & \mathfrak{D}_{k} \ar[r] \ar[d]^{ \iota_{k} } & \mathfrak{F}_{k} \ar[r]  \ar[d]_{ ( \iota_{k} )_{\infty}}& 0 \\
0 \ar[r] & \mathfrak{A} ( \N ) \ar[r] & \mathfrak{A} \ar[r] & \mathfrak{A} ( \infty ) \ar[r] & 0
}
\end{align*}
is commutative with exact rows.  Note that $\mathfrak{D}_{k} = \mathfrak{F}_{k} \oplus_{ \overline{\tau}_{  \mathfrak{A}  } \circ ( \iota_{k} )_{\infty} , \rho_{ \mathfrak{A} } } \ell^{\infty} ( \{ \mathfrak{A} (n) \} )$.   Since $0 \to \mathfrak{A}(\N) \to \mathfrak{A} \to \mathfrak{A}(\infty) \to 0$ is a quasi-diagonal extension by Proposition~\ref{p:quasidiag}, we have that $0 \to \mathfrak{A}(\N ) \to \mathfrak{D}_{k} \to \mathfrak{F}_{k} \to 0$ is a quasi-diagonal extension.

Since $\mathfrak{A}$ is a tight $C^{*}$-algebra over $\widetilde{\N}$, we have that $\mathfrak{A}$ is a continuous $C^{*}$-algebra over $\widetilde{\N}$.  Hence, by Lemma~\ref{l:contsubalg}, $\mathfrak{D}_{k}$ is a continuous $C^{*}$-algebra over $\widetilde{\N}$.  Therefore, by Lemma~\ref{c:continuous} and since $\mathfrak B$ is continuous, $\overline{ \phi}_{\N } \circ \overline{\tau}_{  \mathfrak{A} } \circ ( \iota_{k} )_{\infty}$ and $\overline{\tau}_{\mathfrak{B} } \circ \phi_{\infty} \circ ( \iota_{k} )_{\infty}$ satisfies properties (2) and (3) of Lemma~\ref{l:unitaryequiv}.  Since each $\mathfrak A(n)$ is separable, stable and has real rank zero, $K_{0}(\ell^\infty(\{ \mathfrak A(n)\})) \cong \prod_{ n = 1}^{\infty} K_{0} ( \mathfrak{A}(n) )$ and $K_{0} ( q_{\infty} ( \{ \mathfrak{A}(n) \} ) ) \cong \frac{ \prod_{ n = 1}^{\infty} K_{0} ( \mathfrak{A}(n) ) }{ \bigoplus_{n = 1}^{\infty} K_{0} ( \mathfrak{A}(n) ) }$ where the isomorphisms are induced by the coordinate projections. Hence by Lemma~\ref{l:reducedBusby}, $K_0(\overline \tau_{\mathfrak A})$, and similarly $K_0(\overline \tau_{\mathfrak B})$, is exactly the map induced by the coordinate projections.  Thus, since $K_{*} ( \gamma )$ is a $C( \widetilde{\N} , \Z )$-module homomorphism, we have that $K_{0} ( \overline{\phi}_{\N} \circ \overline{\tau}_{  \mathfrak{A} } )= K_{0} ( \overline{\tau}_{\mathfrak{B}  } \circ  \phi_{\infty} )$.  Therefore, 
\begin{align*}
K_{0} ( \overline{ \phi}_{\N } \circ \overline{\tau}_{  \mathfrak{A}  } \circ ( \iota_{k} )_{\infty} ) = K_{0} ( \overline{\tau}_{\mathfrak{B} } \circ \phi_{\infty} \circ ( \iota_{k} )_{\infty} ).
\end{align*}
By Lemma~\ref{l:unitaryequiv} and Lemma~\ref{l:fullhomo}, there exists an $\widetilde{\N}$-equivariant $*$-homomorphism $\ftn{\psi_{k}}{ \mathfrak{D}_{k} }{\mathfrak{B} }$ such that $E( ( \psi_{k})_{n} ) = E( \phi_{n} ) = \gamma_{n} \circ E((\iota_k)_{n})$ for all $n \in \N$ and $E( ( \psi_{k})_{\infty} ) = E( \phi_{\infty} \circ (\iota_k)_\infty ) = \gamma_{\infty} \circ E((\iota_k)_{\infty})$.  Moreover, by Lemma~\ref{l:fullhomo}, $\psi_{k}$ has the property that for each $p \in \mathfrak{D}_{k}$ with $U = \widetilde{\N} \setminus F$ where $F= \setof{ n \in \widetilde{\N} }{ \pi_{n} ( p ) = 0 }$, we have that $\psi_{k} (p)$ is full in $\mathfrak{B} (U)$.  By Lemma~\ref{l:ethyequal}, the $C( \widetilde{\N} , \Z )$-homomorphisms $K_{*} ( \psi_{k} )$ and $K_{*} (\gamma ) \circ K_{*} ( \iota_{k} )$ are equal.  Therefore, for each projection $p \in \mathfrak{D}_{k}$, we have that $\psi_{k} ( p )$ and $\psi_{k+1} ( \lambda_{k, k+1} ( p) )$ generate the same ideal $\mathfrak{B} (U)$ for some $U \in \widetilde{\N}$ and $[ \psi_{k} ( p ) ] = [ \psi_{k+1} ( \lambda_{k, k+1} ( p) ) ]$ in $K_{0} ( \mathfrak{B} (U)  )$.  Thus, for each projection $p \in \mathfrak{D}_{k}$, $\psi_{k} ( p )$ and $\psi_{k+1} ( \lambda_{k, k+1} ( p) )$ are Murray-von Neumann equivalent since $\mathfrak{B}$ has stable weak cancellation by Proposition~\ref{p:rrstwc}.  Note also that $E( ( \psi_{k} )_{\N} ) = \gamma_{ \N } \circ E( ( \iota_{k} )_{\N} ) =  E ( ( \psi_{k+1} \circ \lambda_{k,k+1} )_{\N} )$.

Let $\mathcal{H}_{k}$ be finite subsets of $\mathfrak{D}_{k}$ such that $\lambda_{k, k+1} ( \mathcal{H}_{k} ) \subseteq \mathcal{H}_{k+1}$ and $\bigcup_{ k = 1}^{\infty} \iota_{k} ( \mathcal{H}_{k} )$ is dense in $\mathfrak{A}$.  By Theorem~\ref{t:uniq}, there exists a unitary $w_{k} \in \multialg{ \mathfrak{B} }$ with $w_{1} = 1_{\multialg{ \mathfrak{B} } }$ such that 
\begin{align*}
\norm{ w_{k+1} ( \psi_{k+1} \circ \lambda_{k,k+1} )( x  ) w_{k+1}^{*} - w_{k}\psi_{k} (x) w_{k} } < \frac{1}{2^{k}}
\end{align*}
for all $x \in \mathcal{H}_{k}$.  Hence, there exists a $*$-homomorphism $\ftn{ \psi }{ \mathfrak{A} }{ \mathfrak{B} }$ such that 
\begin{align*}
\norm{ \psi \circ \iota_{k} ( x  ) - w_{k}\psi_{k}  ( x  ) w_{k}^{*} } < \sum_{ m = k }^{\infty} \frac{1}{2^{m}}
\end{align*}
for all $x \in \mathcal{H}_{k}$.  Since $\psi_{k}$ and $\lambda_{k,k+1}$ are $\widetilde{\N}$-equivariant $*$-homomorphisms, we have that $\psi$ is an $\widetilde{\N}$-equivariant $*$-homomorphism by Lemma~\ref{l:Xlimit}.  By construction, $\psi$ is a full $\widetilde{\N}$-equivariant $*$-homomorphism since for each $x$ in $\mathfrak{D}_{k}$ with $\iota_{k} (x)$ full in $\mathfrak{A} (U)$, we have that $\psi_{k} (x)$ is full in $\mathfrak{B} (U)$.  Also, the $C( \widetilde{\N} , \Z )$-homomorphisms, $K_{*} ( \psi )$ and $K_{*} ( \gamma )$ are equal and $\underline{K} ( \psi_{ n } ) = \underline{K} ( \gamma_{n} )$ for all $n \in \widetilde{\N}$, by Lemma~\ref{l:ethyequal}. 
\end{proof}

\section{Classification using ideal-related $K$-theory with coefficient}

In this section, we prove a classification result for tight $C^{*}$-algebras over $\widetilde{\N}$ whose fibers are AF-algebras or Kirchberg algebras in $\mathcal{N}$ using ideal-related $K$-theory with coefficient.

\begin{lemma}\label{t:totalkthy}
Let $\mathfrak{A}$ and $\mathfrak{B}$ be $C^{*}$-algebras over $\widetilde{\N}$ and let $\alpha$ and $\beta$ be $C( \widetilde{\N} , \Lambda )$-homomorphisms.  Suppose that $K_{0} ( \mathfrak{A}(\infty) )$ is torsion-free, $K_{1} ( \mathfrak{A}( \infty ) )$ is zero, and that the extension 
\begin{align*}
0 \to \mathfrak{A} ( \N ) \to \mathfrak{A} \to \mathfrak{A} ( \infty ) \to 0
\end{align*}
is quasi-diagonal.  If $\alpha_{n} = \beta_{n}$ for all $n \in \widetilde{\N}$ and the $C( \widetilde{\N} , \Z )$-homomorphisms, $K_{*} ( \alpha )$ and $K_{*} ( \beta )$, are equal, then $\alpha = \beta$.
\end{lemma}

\begin{proof}
Since $\alpha_{n} = \beta_{n}$ for all $n \in \widetilde{\N}$, for all $U \subseteq \N$, we have that $\alpha_{U} = \beta_{U}$.  Let $F$ be a finite subset of $\N$ and let $U = \widetilde{\N} \setminus F$.  Set $V = U \setminus \{ \infty \}$.  Note that the extension
\begin{align*}
0 \to \mathfrak{A} ( V ) \to \mathfrak{A} (U) \to \mathfrak{A}(\infty) \to 0
\end{align*}
induces the following commutative diagram
\begin{align*}
\scalebox{.85}{\xymatrix{
K_{1} ( \mathfrak{A}(V) ) \ar[r]^-{ \times n} \ar[d] & K_{1} ( \mathfrak{A}(V) ) \ar[r] \ar[d] & K_{1} ( \mathfrak{A} ( V ) ; \Z_{n} ) \ar[r] \ar[d] & K_{0} ( \mathfrak{A}(V) ) \ar[r]^-{ \times n } \ar[d] & K_{0} ( \mathfrak{A}(V) ) \ar[r] \ar[d] & K_{0} ( \mathfrak{A}(V) ; \Z_{n} ) \ar[d]  \\
K_{1} ( \mathfrak{A}(U) ) \ar[r]^-{ \times n} \ar[d] & K_{1} ( \mathfrak{A}(U) ) \ar[r] \ar[d] & K_{1} ( \mathfrak{A} ( U ) ; \Z_{n} ) \ar[r] \ar[d] & K_{0} ( \mathfrak{A}(U) ) \ar[r]^-{ \times n } \ar[d] & K_{0} ( \mathfrak{A}(U) ) \ar[r] \ar[d] & K_{0} ( \mathfrak{A}(U) ; \Z_{n} ) \ar[d] \\
K_{1} ( \mathfrak{A}(\infty) ) \ar[r]^-{ \times n} & K_{1} ( \mathfrak{A}(\infty) ) \ar[r] & K_{1} ( \mathfrak{A} ( \infty ) ; \Z_{n} ) \ar[r] & K_{0} ( \mathfrak{A}(\infty) ) \ar[r]^-{ \times n } & K_{0} ( \mathfrak{A}(\infty) ) \ar[r] & K_{0} ( \mathfrak{A}(\infty) ; \Z_{n} ) 
}
}
\end{align*}
where the rows and columns are exact sequences.

Since $K_{0} ( \mathfrak{A}(\infty) )$ is torsion-free and $K_{1} ( \mathfrak{A}( \infty ) )$ is zero, we have that $K_{1} ( \mathfrak{A} ( \infty ) ; \Z_{n} ) = 0$.  Therefore, the homomorphism from $K_{1} ( \mathfrak{A}(V) ; \Z_{n} )$ to $K_{1} ( \mathfrak{A}(U) ; \Z_{n} )$ is surjective.  Since $\alpha_{V, 1} = \beta_{V,1}$, we have that $\alpha_{U, 1} = \beta_{U, 1}$.  

Since the extension
\begin{align*}
0 \to \mathfrak{A} ( V ) \to \mathfrak{A} (U) \to \mathfrak{A}(\infty) \to 0
\end{align*}
is quasi-diagonal, we have that the homomorphism from $K_{0} ( \mathfrak{A}(U) )$ to $K_{0} ( \mathfrak{A} (\infty) )$ is surjective.  Exactness of the bottom row and the fact that $K_{1} ( \mathfrak{A}(\infty ) ) = 0$ implies that the homomorphism from $K_{0} ( \mathfrak{A}(\infty) )$ to $K_{0} ( \mathfrak{A}(\infty) ; \Z_{n} )$ is surjective.  A diagram chase now shows that for all $x \in K_{0} ( \mathfrak{A}(U) ; \Z_{n} )$ there exist $y_{1} \in K_{0} ( \mathfrak{A}(U) )$ and $y_{2} \in K_{0} ( \mathfrak{A}(V) ; \Z_{n} )$ such that 
\begin{align*}
x = z_{1} + z_{2},
\end{align*}
where $z_{1}$ is the image of $y_{1}$ under the homomorphism from $K_{0} ( \mathfrak{A} (U) )$ to $K_{0} ( \mathfrak{A}(U) ; \Z_{n} )$ and $z_{2}$ is the image of $y_{2}$ under the homomorphism from $K_{0} ( \mathfrak{A} (V) ; \Z_{n} )$ to $K_{0} ( \mathfrak{A}(U) ; \Z_{n} )$.  Since $\alpha_{V,0} ( y_{2} ) = \beta_{V,0} (y_{2})$ and $K_{0} ( \alpha_{U} ) ( y_{1} ) = K_{0} ( \beta_{U} )( y_{2} )$, we have that 
\begin{align*}
\alpha_{U,0} ( x ) = \alpha_{U,0} ( z_{1} + z_{2} ) = \beta_{U, 0} ( z_{1} + z_{2} ) = \beta_{U, 0} (x).
\end{align*}  
Hence, $\alpha_{U} = \beta_{U}$.
\end{proof}

We are now ready to prove our first main classification result.  In the case that all the fibers are Kirchberg algebras, by \cite[Example~6.14]{mdrm:ethy}, ideal-related $K$-theory without coefficient is not a complete invariant for classification.  The result below shows that ideal-related $K$-theory with coefficients is a complete invariant.

\begin{theor}\label{t:class}
Let $\mathfrak{A}$ and $\mathfrak{B}$ be tight $C^{*}$-algebras over $\widetilde{\N}$.  Suppose for each $n \in \widetilde{\N}$, that $\mathfrak{A}(n)$ is an AF-algebra or a Kirchberg algebra in $\mathcal{N}$ and that $\mathfrak{B}(n)$ is an AF-algebra or a Kirchberg algebra in $\mathcal{N}$.

Suppose that there exists a $C( \widetilde{\N} , \Lambda )$-isomorphism $\ftn{ \gamma }{ \underline{K} ( \mathfrak{A}  ) }{ \underline{K} ( \mathfrak{B} ) }$ such that $K_{0} ( \gamma_{n} )$ is an order isomorphism for each $n \in \widetilde{\N}$.
\begin{itemize}
\item[(1)] If $\mathfrak{A}$ and $\mathfrak{B}$ are stable $C^{*}$-algebras, then there exists an $\widetilde{\N}$-equivariant $*$-isomorphism $\ftn{ \phi }{ \mathfrak{A}  }{ \mathfrak{B}  }$ such that the $C( \widetilde{\N} , \Lambda )$-isomorphisms $\underline{K} ( \phi )$ and $ \gamma $ are equal.

\item[(2)] If $\mathfrak{A}$ and $\mathfrak{B}$ are unital $C^{*}$-algebras and $K_{0} ( \gamma ) ( [ 1_{\mathfrak{A} } ] ) = [ 1_{\mathfrak{B}} ]$, then there exists an $\widetilde{\N}$-equivariant $*$-isomorphism $\ftn{ \phi }{ \mathfrak{A}  }{ \mathfrak{B}  }$ such that the $C( \widetilde{\N} , \Lambda )$-isomorphisms $\underline{K} ( \phi )$ and $ \gamma $ are equal.
\end{itemize}
\end{theor}

\begin{proof}
We first prove (1) in the case that $\mathfrak{A}(\infty)$ is an AF-algebra.  Note that $\mathfrak{B}(\infty)$ is an AF-algebra since $K_{0} ( \mathfrak{C} )_{+} = K_{0} ( \mathfrak{C} )$ for any Kirchberg algebra $\mathcal{C}$ and $K_{0} ( \mathfrak{D} )_{+} \neq K_{0} ( \mathfrak{D} )$ for any non-zero AF-algebra $\mathfrak{D}$.

By \cite[Theorem~6.11]{mdrm:ethy}, there exists an invertible element $\widetilde{\gamma}$ in $E ( \widetilde{\N} , \mathfrak{A} , \mathfrak{B} )$ lifting $\gamma$.  By Theorem~\ref{t:existence}, there exist full $\widetilde{\N}$-equivariant $*$-homomorphisms $\ftn{ \lambda }{ \mathfrak{A}  }{ \mathfrak{B} }$ and $\ftn{ \beta }{ \mathfrak{B} }{ \mathfrak{A}  }$ such that $\underline{K} ( \lambda_{n} ) = \gamma_{n}$, $\underline{K} ( \beta_{n}  ) = ( \gamma^{-1} )_{n}$ for each $n \in \widetilde{\N}$, the $C( \widetilde{\N} , \Z )$-isomorphisms $K_{*} ( \lambda )$ and $K_{*} ( \gamma )$ are equal, and the $C( \widetilde{\N} , \Z )$-isomorphisms $K_{*} ( \beta )$ and $K_{*} ( \gamma^{-1} )$ are equal.  

By Proposition~\ref{p:rrstwc}, $\mathfrak{A}$ and $\mathfrak{B}$ have stable weak cancellation.  Since $\beta \circ \lambda$ and $\id_{\mathfrak{A}}$ are full $\widetilde{\N}$-equivariant $*$-homomorphisms, we have that for each $a$ full in $\mathfrak{A}(U)$, the element $( \beta \circ \lambda )(a)$ is full in $\mathfrak{A}(U)$.  Let $p$ be a projection in $\mathfrak{A}$.  Then $( \beta \circ \lambda )(p)$ and $p$ generate the same ideal $\mathfrak{A}(U)$ for some $U \in \mathbb{O} ( \widetilde{\N} )$.   By construction, $K_{0} ( ( \beta \circ \lambda)_{U} ) = K_{0} ( \id_{ \mathfrak{A}(U)} )$, which implies that $[ (\beta \circ \lambda )(p) ] = [ p ]$ in $K_{0} ( \mathfrak{A}(U) )$.  Since $\mathfrak{A}$ has stable weak cancellation, we have that $(\beta \circ \lambda) (p)$ is Murray-von Neumann equivalent to $p$.  Similarly, for each projection $q$ in $\mathfrak{B}$, we have that $(\lambda \circ \beta)(q)$ is Murray-von Neumann equivalent to $q$.

Using Theorem~\ref{t:uniq} and an approximate intertwining argument, we get a $*$-isomorphism $\ftn{ \phi }{ \mathfrak{A}  }{ \mathfrak{B} }$.  By Lemma~\ref{l:Xlimit}, we have that $\phi$ is an $\widetilde{\N}$-equivariant $*$-isomorphism.  By construction, the $C( \widetilde{\N} , \Z )$-isomorphisms $K_{*} ( \phi )$ and $K_{*} ( \gamma )$ are equal and $\underline{K} ( \phi_{n} ) = \underline{K}( \gamma_{n} )$ for all $n \in \widetilde{\N}$.  By Lemma~\ref{t:totalkthy},  the $C( \widetilde{\N} , \Lambda )$-isomorphisms $\underline{K} ( \phi )$ and $ \gamma $ are equal.  Thus, we have proved (1) for the case that $\mathfrak{A}(\infty)$ is an AF-algebra.

We now prove (1) for the case that $\mathfrak{A}( \infty )$ is a Kirchberg algebra.  Note that there exists a full embedding $\iota$ from $\mathcal{O}_{2}$ to $\mathfrak{A}( \infty )$ since $\mathfrak{A} ( \infty )$ is a Kirchberg algebra.  Let $\ftn{ \overline{\tau}_{\mathfrak{A}} }{ \mathfrak{A} ( \infty ) }{ q_{\infty} ( \{ \mathfrak{A}(n) \} )}$ be the reduce Busby map.  Since $\mathcal{O}_{2}$ is semiprojective, there exists $N \in \N$, such that $\overline{ \tau }_{\mathfrak{A}} \circ \iota$ lifts to a homomorphism from $\mathcal{O}_{2}$ to $\prod_{ n = N }^{\infty} \mathfrak{A}(n)$.  Since the only $*$-homomorphism from a Kirchberg algebra to an AF-algebra is the zero homomorphism, by Lemma~\ref{l:reducedBusby}, we have that 
\begin{align*}
F = \{ n \in \widetilde{\N} : \text{$\mathfrak{A} (n)$ is an AF-algebra} \}
\end{align*}
is finite and a subset of $\N$.  

Arguing as above we have that $\mathfrak{B}( \infty )$ is a Kirchberg algebra.  Hence,
\begin{align*}
G = \{ n \in \widetilde{\N} : \text{$\mathfrak{B} (n)$ is an AF-algebra} \}
\end{align*}
is finite and subset of $\N$.  Since $K_{*} ( \mathfrak{A}(n) ) \cong K_{*} ( \mathfrak{B} (n) )$ as ordered groups, we have that $G = F$.

We have just shown that $\mathfrak{A} \cong \mathfrak{A}( \widetilde{\N} \setminus F ) \oplus \mathfrak{A} ( F )$ and $\mathfrak{A} \cong \mathfrak{B}( \widetilde{\N} \setminus F ) \oplus \mathfrak{B} ( F )$.  Note that $\mathfrak{A}( \widetilde{\N} \setminus F )$ and $\mathfrak{B}( \widetilde{\N} \setminus F )$ are tight $C^{*}$-algebras over $\widetilde{\N} \setminus F$ whose fibers are Kirchberg algebras in $\mathcal{N}$.  The result now follows from \cite[Theorem~6.11 and 5.4]{mdrm:ethy}, Kirchberg's classification of strongly purely infinite $C^{*}$-algebras \cite{kirchpure}, and Elliott's classification of AF-algebras \cite{af}.  Thus we have proved (1) for the case that $\mathfrak{A}(\infty)$ is a Kirchberg algebra.

Since $\mathfrak{A}$ and $\mathfrak{B}$ have stable weak cancellation, (2) now follows from (1) and \cite[Theorem~3.2]{err:strongclass}. 
\end{proof}

\section{A Universal Coefficient Theorem}\label{s:uct}

In this section we prove a universal coefficient theorem (UCT) for $C^\ast$-algebras over $\widetilde{\mathbb N}$ which allows us to improve our classification result. This will be be done using homological algebra in triangulated categories, as done by Ralf Meyer and Ryszard Nest in \cite{rmrn:homalg}.

\subsection{On $C(\widetilde{\mathbb N},\Z)$-modules}

In order to apply the results in \cite{rmrn:homalg} we need a good description of the projective modules and some results on when modules have projective dimension $1$. Our first results will be done for the more general rings $C(X,R)$ (the ring of locally constant functions from $X$ to $R$) for any totally disconnected, metrizable, compact space $X$ and discrete ring $R$.

\begin{lemma}\label{projC(X,R)}
Let $R$ be a discrete ring and $X$ be a totally disconnected, metrizable, compact space. If $P$ is a projective (left or right) $R$-module and $U$ is a clopen subset of $X$, then $C(U,P)$ is a projective (left or right) $C(X,R)$-module.
\end{lemma}
\begin{proof}
Since $P$ is $R$-projective there is an $R$-module $Q$ and an index set $I$ such that $P \oplus Q \cong \bigoplus_I R$. It is easily verified that
\[
C(U,P) \oplus C(U,Q) \oplus C(X \setminus U, \bigoplus_I R) \cong \bigoplus_I C(X,R)
\]
as $C(X, R)$-modules.  Since $\bigoplus_I C(X,R)$ is a free $C(X,R)$-module, it follows that $C(U,P)$ is projective.
\end{proof}

\begin{lemma}\label{fingenideal}
Let $R$ be a discrete ring and $X$ be a totally disconnected, metrizable, compact space. Then any finitely generated right (resp.~left) ideal $J$ in $C(X,R)$ is of the form $\bigoplus_{j=1}^n C(U_j,I_j)$ where $U_1,\dots , U_n$ is a clopen partition in $X$ and each $I_j$ is a finitely generated right (resp.~left) ideal of $R$.
\end{lemma}
\begin{proof}
We only prove the right ideal case, since the left ideal case is exactly the same. Let $f_1,\dots ,f_m \in C(X,R)$ be generators of $J$. We may find a clopen partition $U_1,\dots ,U_n$ such that each $f_k$ is constant on $U_j$ for each $j$. Hence it makes sense to say that $f_k(U_j)$ are elements of $R$. Define $I_j = f_1(U_j) R + \dots +f_m(U_j)R$ for each $j$, which are finitely generated right ideals in $R$. We get that
\begin{eqnarray*}
J &=& f_1 C(X,R) + \dots + f_m C(X,R) \\
&=& f_1 \left(\bigoplus_{j=1}^n C(U_j,R) \right) + \dots + f_m \left(\bigoplus_{j=1}^n C(U_j,R) \right) \\
&=& \bigoplus_{j=1}^n C(U_j,f_1(U_j) R) + \dots + \bigoplus_{j=1}^n C(U_j,f_m(U_j) R) \\
&=& \bigoplus_{j=1}^n C(U_j,I_j).
\end{eqnarray*}
\end{proof}

Recall, that a ring $R$ is called left (resp.~right) \emph{semihereditary} if every finitely generated left (resp.~right) ideal is a projective left (resp.~right) $R$-module.

\begin{propo}
Let $R$ be a discrete (left or right) semihereditary ring, and $X$ be a totally disconnected, metrizable, compact space. Any projective (left or right) $C(X,R)$-module is isomorphic to a direct sum of modules of the form $C(U,I)$ where $U$ is a clopen subset of $X$ and $I$ is a finitely generated ideal in $R$.
\end{propo}
\begin{proof}
By Lemmas~\ref{projC(X,R)} and \ref{fingenideal}, $C(X,R)$ is (left or right) semihereditary. By \cite{Albrecht:PAMS61}, any projective $C(X,R)$-module is a direct sum of finitely generated ideals in $C(X,R)$. The result now follows from Lemma~\ref{fingenideal}.
\end{proof}

\begin{corol}\label{dirsumbasics}
Any countably generated projective $C(X,\Z)$-module is isomorphic to a countable direct sum of modules of the form $C(U,\Z)$ where $U$ is a clopen subset of $X$.
\end{corol}
\begin{proof}
The ring of integers is semihereditary. The countability criterion follows since an uncountable direct sum of non-zero modules can not be countably generated.
\end{proof}

\begin{lemma}\label{l:directlimtprojectivemod}
Let $\{ M_{n} , \phi_{n, n+1 } \}$ be a directed system of projective (left or right) modules over a ring $R$.  Then $\varinjlim ( M_{n} , \phi_{n,n+1} )$ has projective dimension at most 1.
\end{lemma}

\begin{proof}
Define $\ftn{ \psi }{ \bigoplus_{ n =1}^{\infty} M_{n} }{ \bigoplus_{ n = 1}^{\infty} M_{n} }$ by $\psi ( \{ x_{n} \} ) = ( 0 , \phi_{1,2} (x_{1}) , \phi_{2,3} ( x_{2} ) , \dots )$.  A computation shows that $\id - \psi$ is injective.  Therefore, 
\begin{align*}
0 \to \bigoplus_{ n =1}^{\infty} M_{n} \to \bigoplus_{ n =1}^{\infty} M_{n} \to \coker( \id - \psi ) \to 0
\end{align*}
is a projective resolution of length 1 for $\coker( \id - \psi )$.  Since $\coker( \id - \psi ) \cong \varinjlim ( M_{n} , \phi_{n, n+1} )$, the lemma follows. 
\end{proof}

We now restrict our attention to the ring $C(\widetilde{\mathbb N},\Z)$. Given a $C(\widetilde{\mathbb N},\Z)$-module $M$, let $M_n$ (for $n\in \mathbb N$) be the direct summand generated by cutting down the module with the characteristic function on $\{n\} \subseteq \widetilde{\mathbb N}$. Then $\bigoplus_{n\in \mathbb N} M_n$ is a submodule of $M$ and we denote the quotient by $M_\infty$.

In particular, suppose $\mathfrak A$ is a $C^\ast$-algebra over $\widetilde{\N}$. Then $(K_i(\mathfrak A))_n \cong K_i(\mathfrak A(n))$ naturally for $n \in \N$. Moreover, since $\bigoplus_{n\in \mathbb N} K_i(\mathfrak A(n)) \cong K_i(\mathfrak A(\N))$ naturally, and the map $K_i(\mathfrak A(\N)) \to K_i(\mathfrak A)$ is induced by the coordinate inclusions, and is thus injective, it follows that $K_i(\mathfrak A)_\infty \cong K_i(\mathfrak A(\infty))$ naturally.

\begin{propo}\label{projdim1}
Let $M$ be a countably generated $C(\widetilde{\mathbb N},\Z)$-module. If $M_\infty$ is torsion-free as an abelian group then $M$ has projective dimension less than $1$.
\end{propo}
\begin{proof}
For any $n\in \mathbb N$, the module $M_n$ has projective dimension less than $1$. To see this, let $0 \to P_1 \to P_0 \to M_n \to 0$ be a length $1$ projective resolution of the abelian groups for $M_n$. Then the induced sequence $0 \to C(\{n\},P_1) \to C(\{n\},P_0) \to M_n \to 0$ is a $C(\widetilde{\mathbb N},\Z)$-projective resolution of length $1$. Hence the module $\bigoplus_{n\in \mathbb N} M_n$ has projective dimension less than $1$.

Since $M$ is countably generated, so is the abelian group $M_\infty$. Hence $M_\infty$ can be written as an inductive limit of a system of finitely generated free abelian groups, say $\Z^{N_1} \xrightarrow{f_1} \Z^{N_2} \xrightarrow{f_2} \Z^{N_3} \to \dots$. For each natural number denote by $[n,\infty]$ the clopen set $\{ n , n+1,\dots ,\infty\} \subseteq \widetilde{\mathbb N}$. Let $i_n : C([n,\infty],\Z) \to C([n+1,\infty],\Z)$ be the canonical projection. Consider the inductive system of $C(\widetilde{\mathbb N},\Z)$-modules
\[
\Z^{N_1}\otimes_\Z C([1,\infty],\Z) \xrightarrow{f_1\otimes i_1} \Z^{N_2} \otimes_\Z C([2,\infty],\Z) \xrightarrow{f_2\otimes i_2} \Z^{N_3} \otimes_\Z C([3,\infty],\Z) \to \dots.
\]
The direct limit of the system is easily seen to be $M_\infty$. Since $\Z^{N_n} \otimes_\Z C([n,\infty],\Z)$ is isomorphic to a direct sum of modules $C([n,\infty],\Z)$ which are projective by Lemma~\ref{projC(X,R)}, $M_\infty$ is an inductive limit of projective modules, and by Lemma~\ref{l:directlimtprojectivemod}, has projective resolution of length less than $1$.

Now $M$ is an extension of $M_\infty$ by $\bigoplus_{n\in \mathbb N} M_n$, where $M_\infty$ has projective dimension less than $1$. Since $\bigoplus_{n\in \mathbb N} M_n$ has projective dimension less than $1$ it follows easily from the horseshoe lemma (see e.g. \cite[Lemma~2.2.8]{Weibel:CamPress}), that $M$ has projective dimension less than 1.
\end{proof}

\subsection{The UCT}
Let $\mathfrak E(X)$ denote the $E(X)$-theory category with objects being separable $C^\ast$-algebras over $X$ and morphisms from $\mathfrak A$ to $\mathfrak B$ being the elements of $E_0(X;\mathfrak A,\mathfrak B)$. It is proven in \cite{mdrm:ethy} that $\mathfrak E(X)$ is a triangulated category with exact triangles isomorphic to diagrams arising from extensions of $C^\ast$-algebras over $X$. 

\begin{defin}
We define the (classical) \emph{$E$-theoretic bootstrap class} $\mathcal B_E$ to be the $\aleph_0$-localising subcategory of $\mathfrak E$ generated by $\mathbb C$.

We define the \emph{$E(X)$-theoretic bootstrap class} $\mathcal B_E(X)$ to be the full subcategory of $\mathfrak E(X)$ of objects $\mathfrak A$ such that $\mathfrak A(U)$ is an object of $\mathcal B_E$ for any open subset $U$ of $X$.
\end{defin}

The above definition is made in \cite{mdrm:ethy} for any second countable, sober space $X$, but we will only be considering the case where $X$ is a totally disconnected, metrizable, compact space.  We will prove the following.

\begin{theor}\label{UCT}
Let $X$ be a totally disconnected, metrizable, compact space, and let $\mathfrak A$ be an object of $\mathcal B_E(X)$. If $K_\ast(\mathfrak A)$ has $C(X,\Z)$-projective dimension 1, then for any separable $C^\ast$-algebra $\mathfrak B$ over $X$ there is a short exact sequence
\[
0 \to \mathrm{Ext}_{C(X,\Z)} (K_\ast(\mathfrak A), K_{\ast +1} (\mathfrak B)) \to E(X;\mathfrak A,\mathfrak B) \to \Hom_{C(X,\Z)}(K_\ast(\mathfrak A) , K_\ast(\mathfrak B)) \to 0
\]
which is natural in both variables.
\end{theor}

The main part of the proof of the above theorem is contained in Proposition~\ref{universalfunctor} below. We refer the reader to \cite{rmrn:homalg} for the relevant definition.

We will consider the $K$-theory as a covariant functor $K_\ast \colon \mathfrak E(X) \to \Mod_{C(X,\Z)}^{\Z_2,c}$. Here $\Mod_{C(X,\Z)}^{\Z_2,c}$ denotes the category of countably generated $\Z_2$-graded $C(X,\Z)$-modules with evenly graded morphisms.  This category is stable with suspension automorphism functor $\Sigma$ which interchanges the $\Z_2$-grading.

\begin{propo}\label{universalfunctor}
Let $X$ be a totally disconnected, metrizable, compact space and let $\mathfrak I = \ker K_\ast$. Then $\mathfrak E(X)$ has enough $\mathfrak I$-projective objects and $K_\ast \colon \mathfrak E(X) \to \Mod_{C(X,\Z)}^{\Z_2,c}$ is the universal $\mathfrak I$-exact, stable, homological functor.
\end{propo}
\begin{proof}
We will use \cite[Proposition 3.39 and Remark 3.42]{rmrn:homalg}. Note that $K_\ast \colon \mathfrak E(X) \to \Mod_{C(X,\Z)}^{\Z_2,c}$ is clearly an $\mathfrak I$-exact, stable, homological functor, and $\Mod_{C(X,\Z)}^{\Z_2,c}$ is a stable, abelian category with enough projective objects. 

Let $U\subseteq X$ be clopen. Any homomorphism of $C(X,\Z)$-modules $C(U,\Z) \to M$ is uniquely defined by an element in $M\delta_U$, where $\delta_U$ is the characteristic function on $U$. Hence if $\mathfrak B$ is a $C^\ast$-algebra over $X$, then
\[
\Hom_{C(X,\Z)} (\Sigma^j(C(U,\Z),0),K_\ast(\mathfrak B)) \cong K_j(\mathfrak B) \delta_U \cong K_j(\mathfrak B(U)) \cong E_j(X;C(U),\mathfrak B),
\]
for $j=0,1$, where the last (natural) isomorphism follows from \cite[Lemma 2.30]{mdrm:ethy}.  By Corollary~\ref{dirsumbasics}, it follows that any $\Z_2$-graded projective $C(X,\Z)$-module $P$ is of the form
\[
P \cong \bigoplus_{i\in I} (C(U_i,\Z),0) \oplus \bigoplus_{j\in J} (0,C(U_j,\Z))
\]
where each $U_i$ and $U_j$ are clopen subsets of $X$ and $I$ and $J$ are countable index sets. For any such projective module and any $C^\ast$-algebra $\mathfrak B$ over $X$ we get that
\begin{eqnarray*}
&& \Hom_{C(X,\Z)}(P,K_\ast(\mathfrak B))\\
&\cong& \prod_{i\in I} \Hom_{C(X,\Z)}((C(U_i,\Z),0),K_\ast(\mathfrak B)) \oplus \prod_{j\in J} \Hom_{C(X,\Z)}((0,C(U_j,\Z)),K_\ast(\mathfrak B)) \\
&\cong& \prod_{i\in I} E(X;C(U_i),\mathfrak B) \oplus \prod_{j\in J} E_1(X;C(U_j),\mathfrak B) \\
& \cong & E(X;\bigoplus_{i\in I} C(U_i) \oplus \bigoplus_{j\in J}C(U_j,C_0(\R)),\mathfrak B),
\end{eqnarray*}
by countable additivity in the first variable of the $E(X)$ bifunctor. Hence there is a partially defined left adjoint $K_\ast^\dagger$ of $K_\ast$, which is defined on the full subcategory of projective modules. Obviously $K_\ast \circ K_\ast^\dagger ( P) \cong P$ for any countably generated $\Z_2$-graded projective $C(X,\Z)$-module $P$, and thus it follows from \cite[Proposition 3.39 and Remark 3.42]{rmrn:homalg} that $K_\ast$ is the universal $\mathfrak I$-exact stable homological functor.
\end{proof}

\begin{proof}[Proof of Theorem~\ref{UCT}]
By Proposition~\ref{universalfunctor} and \cite[Theorems 3.41 and 4.4]{rmrn:homalg} it suffices to show that any $C^\ast$-algebra $\mathfrak A$ over $X$ in $\mathcal B_E(X)$ is in the localising subcategory generated by the $\mathfrak I$-projective objects. To see this, note that a simple bootstrapping argument implies that $E(X;\mathfrak A,\mathfrak B)=0$ for every $\mathfrak I$-contractible $\mathfrak B$ if and only if $\mathfrak A$ is in the localising subcategory in $\mathfrak E(X)$ generated by $\mathfrak I$-projective objects. From \cite[Propositions 6.10 and 6.5]{mdrm:ethy} we get that $\mathfrak A$ is $E(X)$-equivalent to a direct limit of $C^\ast$-algebras over $X$ of the form $\bigoplus_{j=1}^n C(U_j,\mathfrak A_k)$ where $U_1,\dots ,U_n$ is a clopen partition of $X$ and each $\mathfrak A_k$ is in $\mathcal B_E$ with finitely generated $K$-theory. But since these are obviously in the localising subcategory generated by the $\mathfrak I$-projective objects, and this is closed under taking direct limits, the result follows.
\end{proof}

\begin{corol}\label{c:uct}
Let $\mathfrak A$ and $\mathfrak B$ be separable $C^\ast$-algebras over $\widetilde{\mathbb{N}}$ where $\mathfrak A$ is an object of the bootstrap class $\mathcal B_E(\widetilde{\mathbb{N}})$. If $K_\ast(\mathfrak A( \infty))$ is torsion-free then there is a short exact sequence
\[
0 \to \Ext_{C(\widetilde{\mathbb{N}},\Z)} (K_\ast(\mathfrak A), K_{\ast +1} (\mathfrak B)) \to E(\widetilde{ \mathbb{N}};\mathfrak A,\mathfrak B) \to \Hom_{C(\widetilde{ \mathbb{N}},\Z)}(K_\ast(\mathfrak A) , K_\ast(\mathfrak B)) \to 0
\]
which is natural in both variables. Moreover, if $\mathfrak B$ is also an object of $\mathcal B_E(\widetilde{ \mathbb{N}})$ for which $K_\ast(\mathfrak B( \infty))$ is torsion-free, then the short exact sequence splits (unnaturally).
\end{corol}
\begin{proof}
It follows from Proposition~\ref{projdim1} that $K_\ast(\mathfrak A)$ has projective dimension less than 1, and thus the existence of such a UCT follows from Theorem~\ref{UCT}. By using the same method as when proving that the classical UCT is split \cite{UCT}, it suffices to show that $\mathfrak A$ (and similarly $\mathfrak B$) is $E(\widetilde{\N})$-equivalent to a direct sum $\mathfrak A_0 \oplus \mathfrak A_1$ where $K_{1-i}(\mathfrak A_i)=0$. Let $P^1 \to P^0 \to K_\ast(\mathfrak A)$ be a length 1 projective resolution. Since $P^i \cong (P^i_0,0) \oplus (0,P^i_1)$ naturally, we obtain two projective resolutions $\Sigma^i (P^1_i,0) \to \Sigma^i (P^0_i,0) \to \Sigma^i(K_i(\mathfrak A),0)$ for $i=0,1$. We may lift the maps $\Sigma^i (P^1_i,0) \to \Sigma^i (P^0_i,0)$ for $i=0,1$, by using $K_\ast^\dagger$, to morphisms $Q^1_i \to Q^0_i$ in $\mathfrak E(\widetilde{\N})$. Embed these in exact triangles $\Sigma \mathfrak A_i \to Q^1_i \to Q^0_i \to \mathfrak A_i$. By construction $K_\ast(\mathfrak A_0) = (K_0(\mathfrak A),0)$ and $K_\ast(\mathfrak A_1) = (0,K_1(\mathfrak A))$. Thus $K_\ast(\mathfrak A_0 \oplus \mathfrak A_1) \cong K_\ast(\mathfrak A)$ and since both $\mathfrak A_0 \oplus \mathfrak A_1$ and $\mathfrak A$ satisfy the above UCT, naturality of the UCT and the five lemma implies that the algebras are $E(\widetilde{\N})$-equivalent.
\end{proof}

\section{Classification using ideal-related $K$-theory and application for graph $C^{*}$-algebras}

In this section, we use the universal coefficient theorem established in Section~\ref{s:uct} (Corollary~\ref{c:uct}) and Theorem~\ref{t:class}, to prove a classification result using ideal-related $K$-theory for tight $C^{*}$-algebras over $\widetilde{\N}$ whose fibers are AF-algebras or Kirchberg algebras in $\mathcal{N}$ and the $K$-theory of the $\infty$ fiber is torsion-free.  We then apply our result to graph $C^{*}$-algebras with primitive ideal space $\widetilde{\N}$.  

\begin{theor}\label{t:class}
Let $\mathfrak{A}$ and $\mathfrak{B}$ be tight $C^{*}$-algebras over $\widetilde{\N}$.  Suppose for each $n \in \widetilde{\N}$, that $\mathfrak{A}(n)$ is an AF-algebra or a Kirchberg algebra in $\mathcal{N}$ and that $\mathfrak{B}(n)$ is an AF-algebra or a Kirchberg algebra in $\mathcal{N}$, and suppose that $K_{0} ( \mathfrak{A}(\infty) )$ and $K_{1} ( \mathfrak{A}(\infty) )$ are torsion-free abelian groups.

Suppose that there exists a $C( \widetilde{\N} , \Z )$-isomorphism $\ftn{ \gamma }{ K_{*} ( \mathfrak{A}  ) }{ K_{*}( \mathfrak{B}  ) }$ such that $K_{0} ( \gamma_{n} )$ is an order isomorphism for all $n \in \widetilde{\N}$.  
\begin{itemize}
\item[(1)]  If $\mathfrak{A}$ and $\mathfrak{B}$ are stable $C^{*}$-algebras, then there exists an $\widetilde{\N}$-equivariant $*$-isomorphism $\ftn{ \phi }{ \mathfrak{A}  }{ \mathfrak{B}  }$ such that the $C( \widetilde{\N} , \Z )$-isomorphisms $K_{*} ( \phi )$ and $K_{*} ( \gamma )$ are equal.

\item[(2)]  If $\mathfrak{A}$ and $\mathfrak{B}$ are unital $C^{*}$-algebras and $K_{0} ( \gamma ) ( [ 1_{ \mathfrak{A} } ] ) = [ 1_{ \mathfrak{B} } ]$, then there exists an $\widetilde{\N}$-equivariant $*$-isomorphism $\ftn{ \phi }{ \mathfrak{A}  }{ \mathfrak{B}  }$ such that the $C( \widetilde{\N} , \Z )$-isomorphisms $K_{*} ( \phi )$ and $K_{*} ( \gamma )$ are equal.
\end{itemize}
\end{theor}

\begin{proof}
By Corollary~\ref{c:uct}, there exists a $C( \widetilde{ \N } , \Lambda )$-isomorphism $\ftn{ \widetilde{ \gamma } }{ \underline{K} ( \mathfrak{A}  ) }{ \underline{K} ( \mathfrak{B} ) }$ lifting $\gamma$.  The theorem now follows from Theorem~\ref{t:class}. 
\end{proof}

In every known classification theorem for graph $C^{*}$-algebras, it is always assumed that there are finitely many gauge-invariant ideals.  It turns out that we may use Theorem~\ref{t:class} to classify graph $C^{*}$-algebras with infinitely many gauge-invariant ideals.  This is due to the fact that graph $C^{*}$-algebras that are tight over $\widetilde{\N}$ has a special form covered by Theorem~\ref{t:class}.  This was proved in \cite{jg:t1graph}.  For the definition of graph $C^{*}$-algebras, see \cite{FowlerLacaRaeburn:PAMS00}.

\begin{theor}\label{t:classgraph}
Let $\mathfrak{A}$ and $\mathfrak{B}$ be graph $C^{*}$-algebras that are tight $C^{*}$-algebras over $\widetilde{\N}$.  

\begin{itemize}
\item[(1)] If there exists a $C( \widetilde{\N} , \Lambda )$-isomorphism $\ftn{ \gamma }{ \underline{K} ( \mathfrak{A} \otimes \K  ) }{ \underline{K} ( \mathfrak{B} \otimes \K ) }$ such that $K_{0} ( \gamma_{n} )$ is an order isomorphism for each $n \in \widetilde{\N}$, then there exists an $\widetilde{\N}$-equivariant $*$-isomorphism $\ftn{ \phi }{ \mathfrak{A} \otimes \K  }{ \mathfrak{B} \otimes \K }$ such that the $C( \widetilde{\N} , \Lambda )$-isomorphisms $\underline{K} ( \phi )$ and $ \gamma $ are equal.

\item[(2)] If there exists a $C( \widetilde{\N} , \Z )$-isomorphism $\ftn{ \gamma }{ K_{*} ( \mathfrak{A} \otimes \K  ) }{ K_{*} ( \mathfrak{B} \otimes \K ) }$ such that $K_{0} ( \gamma_{n} )$ is an order isomorphism for each $n \in \widetilde{\N}$, then there exists an $\widetilde{\N}$-equivariant $*$-isomorphism $\ftn{ \phi }{ \mathfrak{A} \otimes \K }{ \mathfrak{B} \otimes \K }$ such that the $C( \widetilde{\N} , \Z )$-isomorphisms $K_{*}( \phi )$ and $ \gamma$ are equal.
\end{itemize}
\end{theor}

\begin{proof}
Since $\mathfrak{A}$ is a graph $C^{*}$-algebra and $\mathfrak{A}(n)$ is simple, we have that $\mathfrak{A}(n)$ is either an AF-algebra or a Kirchberg algebra in $\mathcal{N}$ for all $n \in \N$.  By \cite[Remark~4]{jg:t1graph}, we have that $\mathfrak{A}(\infty)$ is an AF-algebra.  Similarly, $\mathfrak{B}(n)$ is either an AF-algebra or a Kirchberg algebra in $\mathcal{N}$ for all $n \in \N$ and $\mathfrak{B} ( \infty )$ is an AF-algebra.  (1) now follows from Theorem~\ref{t:class} and (2) follows from Theorem~\ref{t:class}.
\end{proof}

If $\mathfrak{A}$ is a graph $C^{*}$-algebra that is a tight $C^{*}$-algebra over $\widetilde{\N}$, then it is necessarily non-unital since every unital graph $C^{*}$-algebra with real rank zero has finitely many ideals.  Theorem~\ref{t:classgraph} gives a strong stable classification theorem.  The question of what additional information is needed to get a strong classification theorem for non-stable graph $C^{*}$-algebras that are tight $C^{*}$-algebras over $\widetilde{\N}$ remains open.

\section*{Acknowledgement}
This work was supported by the Danish National Research Foundation through the Centre for Symmetry and Deformation (DNRF92) and by a grant from the Simons Foundation (\#279369 to Efren Ruiz).
The first author would like to express his gratitude to the Department of Mathematics at University of Hawaii, Hilo for their kindness during his visit in Spring 2013 where most of this work was carried out.  The second author also want to thank the Department of Mathematical Sciences at University of Copenhagen for their hospitality during his visit in Fall 2013.

\def\cprime{$'$}

\end{document}